\documentclass[12pt,leqno,a4paper]{amsart}

\usepackage{amsthm,amssymb,verbatim}
\usepackage{array,arydshln}
\usepackage{hyperref}              
\usepackage{verbatim} 
\usepackage{enumitem}

\usepackage{comment}
\usepackage[cmtip, arrow]{xy} 

\input xy
\xyoption{all}

\entrymodifiers={!!<0pt,0.7ex>+}  

\hypersetup{pdftex,                           
bookmarks=true,
pdffitwindow=true,
colorlinks=true,
citecolor=black,
filecolor=black,
linkcolor=black,
urlcolor=blue,
hypertexnames=true}

\theoremstyle{plain}
\newtheorem{Theorem}{Theorem}[section]

\newtheorem{Corollary}[Theorem]{Corollary}
\newtheorem{Lemma}[Theorem]{Lemma}
\newtheorem{Proposition}[Theorem]{Proposition}

\theoremstyle{definition}

\newtheorem{Remark}[Theorem]{Remark}

\newtheorem{Notation/Definition}[Theorem]{Notation/Definition}

\def\cF{{\mathcal{F}}}

\def\cO{\mathcal{O}}

\def\fA{{\mathfrak{A}}}

\def\Aut{\mathrm{Aut}}

\def\dim{\mathrm{dim}}

\def\Gal{\mathrm{Gal}}
\def\Hom{\mathrm{Hom}}

\def\Irr{\mathrm{Irr}}           
\def\IBr{\mathrm{IBr}}

\def\Tr{\mathrm{Tr}}

\DeclareMathOperator{\Syl}{Syl}

\DeclareMathOperator{\GL}{\operatorname{GL}}
\DeclareMathOperator{\SL}{\operatorname{SL}}
\DeclareMathOperator{\PSL}{\operatorname{PSL}}
\DeclareMathOperator{\PGL}{\operatorname{PGL}}

\begin{document}

\title{splendid Morita equivalences for principal 2-blocks with dihedral defect groups}

\author{Shigeo Koshitani}
\address{Shigeo Koshitani\\ Center for Frontier Science,
Chiba University, 1-33 Yayoi-cho, Inage-ku, Chiba, 263-8522, Japan.}
\email{koshitan@math.s.chiba-u.ac.jp}

\author{Caroline Lassueur}
\address{Caroline Lassueur\\ 
FB Mathematik, TU Kaiserslautern, 
Postfach 3049, 67653 Kaisers\-lautern, Germany.}
\email{lassueur@mathematik.uni-kl.de}

\date{\today}
\thanks{The first author was supported by the Japan Society for 
Promotion of Science (JSPS), Grant-in-Aid for Scientific Research
(C)15K04776, 2015--2018. The second author acknowledges financial support by the TU Nachwuchsring of the TU Kaiserslautern as well as by DFG SFB TRR 195.}

\maketitle

\pagestyle{myheadings}
\markboth{S. Koshitani and C. Lassueur}{Splendid Morita equivalences for principal 2-blocks with dihedral defect groups}


\begin{abstract} 
Given a dihedral $2$-group $P$ of order at least~8, we classify the splendid Morita equivalence classes of principal $2$-blocks with defect groups isomorphic to $P$.
To this end  we construct explicit stable  equivalences of Morita type induced by specific Scott modules using Brauer indecomposability and gluing methods; we then determine when these stable equivalences are actually Morita equivalences, and hence automatically splendid Morita equivalences. Finally, we compute the generalised decomposition numbers in each case.
\end{abstract}

\maketitle

\section{Introduction}

In this paper, we are concerned with the classification of principal $2$-blocks 
with dihedral defect groups of order at least $8$, up to {\it splendid Morita equivalence}, also often called {\it Puig equivalence}.
\par
This is motivated by a conjecture of Puig's \cite{Pui82} known as Puig's Finiteness Conjecture (see Brou{\'e} \cite[6.2]{Bro94} or Th\'{e}venaz \cite[(38.6) Conjecture]{The95} for published versions) stating that for a given prime $p$ and a finite $p$-group $P$ there are only finitely many isomorphism classes of interior $P$-algebras arising as source algebras of $p$-blocks of finite groups with defect groups isomorphic to $P$, or equivalently that there are only a finite number of {\it splendid Morita equivalence}  classes of blocks of finite groups with defect groups isomorphic to~$P$. This obviously strengthens Donovan's Conjecture.  However, we emphasise that by contrast to Donovan's Conjecture,  if $p$ is a prime number, $(K,\cO, k)$ a $p$-modular system with $k$ algebraically closed, and  Puig's Finiteness Conjecture 
holds over $k$, then it automatically holds over~$\cO$, since the bimodules inducing splendid Morita equivalences are liftable from $k$ to $\cO$.
\par
The cases where $P$ is either cyclic \cite{Lin96b} or a Klein-four group \cite{CEKL11} are the only cases  where this conjecture has been proved to hold in full generality. Else, under additional assumptions, Puig's Finiteness Conjecture has also been proved for several classes of finite groups, as for instance for $p$-soluble groups \cite{Pui94}, for symmetric groups \cite{Pui94}, for alternating groups \cite{Kes02}, of the double covers thereof \cite{Kes96},  for Weyl groups \cite{Kes00}, or for classical groups \cite{HK00, HK05, Kes01}.
The next cases to investigate should naturally be in tame representation type.
\par
In this paper, we investigate the principal blocks of groups $G$ with a Sylow $2$-subgroup~$P$ 
which is dihedral of order at least $8$ over an algebraically closed field $k$ of characteristic~$2$. Our main result is the following.

\begin{Theorem}\label{MainTheorem} 
Let $n\geq 3$ be a positive integer. Then the splendid Morita equivalence classes of principal $2$-blocks of finite groups with
dihedral defect group $D_{2^n}$ of order $2^n$ coincide with the Morita equivalence classes of such blocks.
More accurately, a principal block with dihedral defect group $D_{2^n}$ is splendidly Morita equivalent to precisely one of the following blocks:
\begin{enumerate}[leftmargin=1.3cm]
\item[\rm(1)] $kD_{2^n}$;
\item[\rm(2)] $B_0(k\mathfrak A_7)$ in case $n=3$;
\item[\rm(3)] $B_0(k[{\mathrm{PSL}}_2(q)])$, where $q$ is a fixed odd prime power such that $(q-1)_2=2^n$;
\item[\rm(4)] $B_0(k[{\mathrm{PSL}}_2(q)])$, where $q$ is a fixed odd prime power such that $(q+1)_2=2^n$;
\item[\rm(5)] $B_0(k[{\mathrm{PGL}}_2(q)])$, where $q$ is a fixed odd prime power such that $2(q-1)_2=2^n$; or
\item[\rm(6)] $B_0(k[{\mathrm{PGL}}_2(q)])$, where $q$ is a  fixed odd prime power such that $2(q+1)_2=2^n$.
\end{enumerate}
In particular, if $q$ and $q'$ are two odd prime powers as in {\rm(3)}-{\rm(6)} such that either \linebreak $(q-1)_2=(q'-1)_2$ (Cases {\rm(3)} and {\rm(5)}), or $(q+1)_2=(q'+1)_2$ (Cases {\rm(4)} and {\rm(6)}), then $B_0(k[{\mathrm{PSL}}_2(q)])$ is splendidly Morita equivalent to $B_0(k[{\mathrm{PSL}}_2(q')])$ and $B_0(k[{\mathrm{PGL}}_2(q)])$ is splendidly Morita equivalent to $B_0(k[{\mathrm{PGL}}_2(q')])$. 
\end{Theorem}

\begin{Remark}
We note that if $G$ is a soluble group and $B$ is an arbitrary $2$-block of $G$ 
with a defect group $P\cong D_{2^n}$ with $n\geq 3$ which is not  nilpotent,   
then $n=3$ and $B$ is actually splendidly Morita equivalent to $k\mathfrak S_4\cong k[{\mathrm{PGL}}_2(3)]$ 
(see \cite{Kos82}).
There is also an interesting and  related result by Linckelmann \cite{Lin94} where all
derived equivalence classes of blocks $B$ with dihedral defect groups over the field $k$ are determined.
\end{Remark}

Furthermore, we will prove in Corollary~\ref{cor:steqD2n} that, for a given defect group $P\cong D_{2^n}$ ($n\geq 3$), up to stable equivalence of Morita type, there are exactly three equivalence classes of principal blocks of finite groups $G$  with defect group $P$, and these depend only on the fusion system $\cF_P(G)$, or equivalently on the number of modular simple modules in $B_0(kG)$.\\

In order to prove Theorem~\ref{MainTheorem}, we will construct explicit Morita equivalences induced by bimodules given by Scott modules of the form $\mathrm{Sc}(G\times G',\Delta P)$. First we will construct stable equivalences of Morita type using these modules using gluing methods and then determine when these stable equivalences are actually Morita equivalences.
To reach this aim, we make use of the notion of Brauer indecomposability, introduced in \cite{KKM11}. In particular, we will use some recent results of Ishioka and Kunugi~\cite{IK17} in order to prove the following theorem:

\begin{Theorem}\label{MainThm1}\label{MainThm}
Let $G$ be a finite group with a dihedral $2$-subgroup $P$ of order at least~$8$. Assume moreover that the fusion system
$\mathcal F_P(G)$ is saturated and $C_G(Q)$ is $2$-nilpotent for every $\mathcal F_P(G)$-fully normalised non-trivial subgroup  $Q$ of $P$.
Then the Scott module ${\mathrm{Sc}}(G,\, P)$ is Brauer indecomposable.
\end{Theorem}

This result, crucial for our work, may in fact be of independent interest as it is an extension of the main results of~\cite{KKL15}. We note that further results on Brauer indecomposability of Scott modules under different hypotheses may be found in \cite[Theorem~1.2]{KKM11}, \cite[Theorem~1.2(b)]{KKL15} and \cite{Tuv14}.
\\

The paper is structured as follows. In Section~\ref{sec:nota} we set up our notation and recall background material which we will use throughout. In Section~\ref{sec:Scott} we establish some properties of Scott modules of direct products with respect to diagonal $p$-subgroups. From Section~\ref{sec:stableeq} onwards, we will assume that the field $k$ has characteristic $2$ and we will focus our attention on groups with dihedral Sylow $2$-subgroups of order at least $8$.
In Section~\ref{sec:stableeq} we prove Theorem~\ref{MainThm1}. In Sections~\ref{sec:pslpgl}~and~\ref{sec:autgrps} we determine when the  stable equivalences constructed in Section~\ref{sec:stableeq} are indeed Morita equivalences. In Section~\ref{sec:proofs} we prove Theorem~\ref{MainTheorem}, and finally, in Section~\ref{sec:genDecNum}, as a consequence of Theorem~\ref{MainTheorem},   
we can specify the signs occurring in Brauer's computation of the generalised decomposition numbers of principal blocks with dihedral defect groups in \cite[\S VII]{Bra66}. This will yield the following result:

\begin{Corollary}
If $G$ is a finite group with a dihedral Sylow $2$-subgroup of order $2^n$ with $n\geq 3$, then $|\Irr(B_0(kG))|=2^{n-2}+3$ and the values at non-trivial $2$-elements of the ordinary irreducible characters in $\Irr(B_0(kG))$ are given by the non-trivial generalised decomposition numbers of $B_0(kG)$ and depend only on the splendid Morita equivalence class of $B_0(kG)$.
\end{Corollary}

Here by \emph{non-trivial} generalised decomposition number, we mean the generalised decomposition numbers parametrised by non-trivial $2$-elements.


\section{Notation and quoted results}\label{sec:nota}

\subsection{Notation} \label{subsec:Nota}{\ }
Throughout this paper, unless otherwise stated we adopt the following notation and conventions.  All groups considered are assumed to be finite and all modules over finite group algebras are assumed to be finitely generated unitary right modules.
We let  $G$ denote a finite group, and $k$  an algebraically closed field of characteristic $p>0$.\par
Given a positive integer $n$,  we write $D_{2^n}, C_n, \mathfrak S_n$ and $\mathfrak A_n$ for the dihedral group of order~$2^n$,
the cyclic group of order $n$,  the symmetric group of degree $n$ and the alternating group of degree $n$, respectively. 
We write $H\leq G$ when $H$ is a subgroup of $G$. Given two finite groups $N$ and $H$, we denote by $N\rtimes H$ 
a semi-direct product of $N$ by $H$ (where $N\vartriangleleft (N\rtimes H))$. 
For a subset $S$ of $G$, we set $S^g:=g^{-1}Sg$, and for $h\in G$ we set $h^{g}:=g^{-1}hg$.
For an integer $n\geq 1$, 
$\langle g_1,\ldots, g_n\rangle$ is the 
subgroup of $G$ generated by
the elements $g_1,\ldots,g_n\in G$. We denote the centre of $G$ by  $Z(G)$ and we set $\Delta G:=\{(g,g)\in G\times G \,|\, g\in G\}\leq G\times G$.
\par
Given a $p$-subgroup $P\leq G$ we denote by $\mathcal F_P(G)$ the fusion system of $G$ on $P$; that is the category whose objects are the $p$-subgroups of $P$, and whose morphisms from $Q$ to~$R$ are the group homomorphisms induced by conjugation by elements of $G$, see \cite[Definition I.2.1]{AKO11}. We recall that if $P$ is a Sylow $p$-subgroup of $G$, then $\mathcal F_P(G)$ is saturated, see \cite[Definition I.2.2]{AKO11}. For further notation and terminology on fusion systems, we refer to \cite{AKO11} and \cite{BLO03}.
\par
The trivial $kG$-module is denoted by $k_G$. If $H\leq G$ is a subgroup, 
$M$ is a $kG$-module and $N$ is a $kH$-module, 
then we write $M^{\ast}:=\Hom_k(M,k)$ for the $k$-dual of $M$,  $M{\downarrow}_H$ for the restriction of $M$ to $H$ and $N{\uparrow}^G$ for the induction of $N$ to $G$. 
Given an $H\leq G$, we denote by $P_H(k_G)$ the $H$-projective cover of the trivial module $k_G$  
and we let $\Omega_H(k_G)$ denote the $H$-relative Heller operator, that is 
$\Omega_H(k_G)= {\mathrm{Ker}}\,( P_H(k_G)
\twoheadrightarrow k_G)$, the kernel of the canonical projection (see \cite{The85}).
We write $B_0(kG)$ for the principal block of $kG$. 
For a $p$-block $B$, we denote by $\Irr(B)$ the set of ordinary irreducible characters in $B$ and by $\IBr(B)$ the set of irreducible Brauer characters in $B$. Further we use the standard notation $k(B):=|\Irr(B)|$ and $l(B):=|\IBr(B)|$.
\par
For a subgroup $H\leq G$ we denote the (Alperin-)Scott $kG$-module with respect to $H$ by ${\mathrm{Sc}}(G,H)$.
By definition ${\mathrm{Sc}}(G,H)$ is the unique indecomposable direct summand
of the induced module ${k_H}{\uparrow}^G$ which contains $k_G$ in its
top (or equivalently in its socle).  If $Q\in \Syl_p(H)$, then $Q$ is a vertex of ${\mathrm{Sc}}(G,H)$ 
and a $p$-subgroup of $G$ is a vertex of ${\mathrm{Sc}}(G,H)$ if and only if it is $G$-conjugate to $Q$. 
It follows that ${\mathrm{Sc}}(G,H)={\mathrm{Sc}}(G,Q)$. We refer the reader to 
\cite[\S2]{Bro85} and  \cite[Chap.4 \S 8.4]{NT88} for these results. Furthermore, we will need the fact that ${\mathrm{Sc}}(G,H)$ is nothing else but the relative $H$-projective cover
$P_H(k_G)$ of the trivial module $k_G$; see \cite[Proposition 3.1]{The85}. In order to produce splendid Morita equivalences between principal 
blocks of two finite groups $G$ and $G'$ with a common defect group $P$, we mainly use Scott modules of the form 
${\mathrm{Sc}}(G\times G', \, \Delta P)$, which are obviously $(B_0(kG),B_0(kG'))$-bimodules by the previous remark.

For further notation and terminology, we refer the reader to the books \cite{Gor68}, \cite{NT88} and \cite{The95}.

\subsection{Equivalences of block algebras} \label{subsec:PuigEq}{\ }
Let $G$ and $H$ be two finite groups, and let $A$ and $B$ be block algebras  of  $kG$ and $kH$
with defect groups $P$ and $Q$, respectively.\par
The algebras $A$ and $B$ are called \emph{splendidly Morita equivalent} (or \emph{Puig equivalent}), if  
there is a Morita equivalence between $A$ and $B$ induced by an $(A,B)$-bimodule $M$ such that $M$, seen 
as a right $k[G\times H]$-module, is a 
$p$-permutation module. In this case, we write $A\sim_{SM} B$. 
Due to a result of Puig (see \cite[Corollary~7.4]{Pui99} and \cite[Proposition~9.7.1]{Lin18}),
the defect groups $P$ and $Q$ are isomorphic (and hence from now on we identify $P$ and $Q$).
Obviously $M$ is indecomposable as a $k(G\times H)$-module.
Further since $_AM$ and $M_B$ are both projective, $M$ has a vertex $R$ which is written as
$R=\Delta (P) \leq G\times H$.  Then, 
this is equivalent to the condition that $A$ and $B$ have source algebras 
which are isomorphic as interior $P$-algebras by the result of Puig and Scott
(see \cite[Theorem~4.1]{Lin01} and \cite[Remark 7.5]{Pui99}).

In particular, we note that if  
the Scott module $M:=\textrm{Sc}(G\times H,\Delta P)$ induces a Morita equivalence between 
the principal blocks $A$ and $B$ of $kG$ and $kH$, respectively,
then this is a splendid Morita equivalence because Scott modules are $p$-permutation modules by definition.\par

Let now $M$ be an $(A,B)$-bimodule and $N$  a $(B,A)$-bimodule. 
Following Brou\'{e} \cite[\S5]{Bro94}, 
we say that the pair $(M,N)$ induces a \emph{stable equivalence of Morita type} between $A$ and~$B$ if $M$ 
and $N$ are projective both as left and right modules, and there are isomorphisms of $(A,A)$-bimodules and $(B,B)$-bimodules
$$M\otimes_{B}N\cong A\oplus X\quad\quad\text{and}\quad\quad N\otimes_{A}M\cong B\oplus Y\,,$$
 respectively, 
where $X$ is a projective $(A,A)$-bimodule and $Y$ is a projective $(B,B)$-bimodule.\\

The following result of Linckelmann will allow us to construct Morita equivalences using stable equivalences of Morita type.

\begin{Theorem}[{}{\cite[Theorem 2.1(ii),(iii)]{Lin96}}]\label{thm:LinckMoritaEq}
Let $A$ and $B$ be finite-dimensional $k$-algebras which are indecomposable non-simple self-injective $k$-algebras. Let $M$ be an $(A,B)$-bimodule inducing a stable equivalence between $A$ and $B$.
\begin{enumerate}
\item[\rm(a)] If $M$ is indecomposable, then for any simple $A$-module $S$, the $B$-module $S\otimes_A M$ is indecomposable and non-projective.
\item[\rm(b)] If for any simple $A$-module $S$, the $B$-module $S\otimes_A M$ is simple, then the functor $-\otimes_A M$ induces a Morita equivalence between $A$ and $B$.
\end{enumerate}
\end{Theorem}

Furthermore, we recall the following fundamental result, originally due to Alperin \cite{Alp76} and Dade \cite{Dad77}, which will provide us with an important source of splendid Morita equivalences in Section~\ref{sec:autgrps}.

\begin{Theorem}[{}{\cite[(3.1) Lemma]{KK02}}]\label{thm:AlperinDade}
Let $\widetilde{G}$ be a finite group. Let $G\trianglelefteq \widetilde{G}$ be a normal subgroup such that $\widetilde{G}/G$ is a $p'$-group and $\widetilde{G}=G\,C_{\widetilde{G}}(P)$, where $P$ is a Sylow $p$-subgroup of $G$. Furthermore, let $\tilde{e}$ and $e$ be the block idempotents corresponding to $B_0(k\widetilde{G})$ and $B_0(kG)$, respectively.  Then the following holds:
\begin{enumerate}
\item[\rm(a)] The map $B_0(kG)\longrightarrow B_0(k\widetilde{G}), a\mapsto a\tilde{e}$ is a $k$-algebra isomorphism, so that $e\tilde{e}=\tilde{e}e=\tilde{e}$.
\item[\rm(b)] The right $k[\widetilde{G}\times G]$-module $B_0(k\widetilde{G})=\tilde{e}k\widetilde{G}=\tilde{e}k\widetilde{G}e$  induces a splendid Morita equivalence between $B_0(k\widetilde{G})$ und $B_0(kG)$.
\end{enumerate}
\end{Theorem}
\medskip

\subsection{The Brauer construction and Brauer indecomposability} \label{subsec:BrauerIndec}{\ }
Given a $kG$-module $V$ and a $p$-subgroup $Q\leq G$, the \emph{Brauer  construction} (or \emph{Brauer quotient}) of $V$ with  respect to $Q$  is defined to be the $kN_G(Q)$-module
$$
V(Q):= V^{Q}\big/ \sum_{R<Q}\Tr^Q_R(V^R)\, ,
$$
where $V^{Q}$ denotes the set of $Q$-fixed points of $V$, and for each proper subgroup $R<Q$, $\Tr^Q_R~:~V^R\longrightarrow V^Q, v\mapsto \sum_{xR\in Q/R}xv$ denotes the relative trace map. See e.g. \cite[\S 27]{The95}. We recall that the Brauer  construction  with  respect  to $Q$ sends  a $p$-permutation $kG$-module $V$ functorially to the $p$-permutation $kN_G(Q)$-module $V(Q)$, see \cite[p.402]{Bro85}.\\

Furthermore, following the terminology introduced in \cite{KKM11}, a $kG$-module $V$ is said to be \emph{Brauer indecomposable} if the $kC_G(Q)$-module $V(Q)\!\downarrow^{N_G(Q)}_{C_G(Q)}$ is indecomposable or zero for each $p$-subgroup $Q\leq G$.
\par
In order to detect Brauer indecomposability, we will use the following two recent results of Ishioka and Kunugi:

\begin{Theorem}[{}{\cite[Theorem 1.3]{IK17}}]\label{thm:IK17Thm1.3}
Let $G$ be a finite group and $P$ a $p$-subgroup of~$G$. Let $M:=\mathrm{Sc}(G,P)$. Assume that the fusion system $\mathcal{F}_P(G)$ is saturated. Then the following assertions are equivalent:
\begin{enumerate}
\item[\rm(i)] $M$ is Brauer indecomposable.
\item[\rm(ii)] $\mathrm{Sc}(N_G(Q),N_P(Q))\!\downarrow^{N_G(Q)}_{QC_G(Q)}$ is indecomposable for each $\mathcal{F}_P(G)$-fully normalised subgroup $Q$ of $P$.
\end{enumerate}
\end{Theorem}

\begin{Theorem}[{}{\cite[Theorem 1.4]{IK17}}]\label{thm:IK17Thm1.4}
Let $G$ be a finite group and $P$ a $p$-subgroup of $G$, $Q$ an $\mathcal{F}_P(G)$-fully normalised subgroup of $P$, and suppose that $\mathcal{F}_P(G)$ is saturated. Assume moreover that there exists a subgroup $H_Q$ of $N_G(Q)$ satisfying the  following two conditions:
\begin{enumerate}
\item[\rm(1)] $N_P(Q)\in\Syl_p(H_Q)$; and 
\item[\rm(2)] $|N_G(Q):H_Q|=p^{a}$ ($a\geq 0$).
\end{enumerate}
Then  $\mathrm{Sc}(N_G(Q),N_P(Q))\!\downarrow^{N_G(Q)}_{QC_G(Q)}$ is indecomposable.
\end{Theorem}
\medskip

\subsection{Principal blocks with dihedral defect groups} \label{sec:O2'trivial}\label{sec:2.4}{\ }
Since $O_{2'}(G)$ acts trivially on the principal block of~$G$ it is well-known that $B_0(G)$ and $B_0(G/O_{2'}(G))$ are 
Morita equivalent. Such a Morita equivalence is induced by the $(B_0(k[G/O_{2'}(G)]), B_0(kG))$-bimodule $B_0(k[G/O_{2'}(G)])$, 
which is obviously a $2$-permutation module. Hence these blocks are indeed splendidly Morita equivalent, and we may restrict our 
attention to the case $O_{2'}(G)=\{1\}$.\\
\medskip 

We recall that if $G$ is a finite group with  a dihedral Sylow $2$-subgroup $P$ of order at least $8$, then   
Gorenstein and Walter \cite{GW65}  
proved that $G/O_{2'}(G)$ is isomorphic to either
\begin{enumerate}
   \item[\rm{(D1)}] $P$,
   \item[\rm{(D2)}] the alternating group $\mathfrak A_7$,  or
   \item[\rm{(D3)}] a subgroup of ${\mathrm{P\Gamma L}}_2(q)$ 
   containing ${\mathrm{PSL}}_2(q)$, where $q$ is a power of an odd prime. In other words, one of the following groups:
   \begin{itemize}
      \item[(i)]  $\PSL_2(q)\rtimes C_f$ where $q$ is a power of an odd prime such that $q\equiv\pm 1\pmod{8}$, and $f\geq1$ is a 
suitable odd number;  or
      \item[(ii)] $\PGL_2(q)\rtimes C_f$ where $q$ is a power of an odd prime and $f\geq1$ is a suitable odd number. 
    \end{itemize}  
\end{enumerate}
The fact that $q$  is a power of an odd prime can be found in  \cite[Chapter 6 (8,9)]{Suz86}. Moreover, 
the splitting of case (D3) into (i) and (ii)
follows from the fact that  
$${\mathrm{P}} \Gamma{\mathrm{L}}_2(q)\cong {\mathrm{PGL}}_2(q)\rtimes 
{\mathrm{Gal}}(\mathbb F_{q}/\mathbb F_{r})\,,$$  where $q=r^m$ is a power of an odd prime~$r$ and the Galois group
${\mathrm{Gal}}(\mathbb F_{q}/\mathbb F_{r})$  
is cyclic of order $m$, generated by the Frobenius automorphism 
$F:\mathbb F_q\longrightarrow \mathbb F_q, x\mapsto x^r$. 
This implies that $C_f$ is a cyclic subgroup of 
${\mathrm{Gal}}(\mathbb F_{q}/\mathbb F_{r})$  
generated by a power of $F$, and moreover the requirement that $P$ is dihedral 
forces $f$ to be odd.  (Here we apply \cite[Chapter~6 (8.9)]{Suz86} implying that $C_f$ is of odd order. See also the beginning of \cite[\S16.3]{Gor68}.)

\subsection{Dihedral $2$-groups and $2$-fusion} \label{subsec:fusion}{\ }
Assume $G$ is a finite group having a Sylow $2$-subgroup~$P$ which is  a dihedral $2$-group $D_{2^n}$ 
of order $2^n$ ($n\geq 3$). Write 
$$P:=D_{2^n}:= \langle s,t \mid  s^{2^{n-1}}=t^2=1, tst=s^{-1} \rangle$$
and set $z:=s^{2^{n-2}}$, so that $\left<z\right>=Z(P)$. Then there are three possible fusion systems $\mathcal{F}_P(G)$ \smallskip   on~$P$. 

\begin{enumerate}
\item[\rm(1)] First case: $\mathcal{F}_P(G)= \mathcal{F}_P(P)$. There are exactly three $G$-conjugacy classes of~involutions in $P$: $\{z\}$, $\{s^{2j}t\mid0\leq j\leq 2^{n-2}-1\}$ and  $\{s^{2j+1}t\mid0\leq j\leq 2^{n-2}-1\}$. Moreover, $l(B_0(kG))=1$, that is $B_0(kG)$ possesses exactly one simple module, namely the trivial  module~\smallskip$k_G$.
\item[\rm(2)] Second case: $\mathcal{F}_P(G)= \mathcal{F}_P(\PGL_2(q))$, where  $2(q\pm1)_2=2^n$. There are exactly two $G$-conjugacy classes of involutions in $P$, represented by the elements $z$ and $st$. Note that $t$ is fused with $z$ in this case. Moreover, \smallskip $l(B_0(kG))=2$.
\item[\rm(3)] Third case: $\mathcal{F}_P(G)= \mathcal{F}_P(\PSL_2(q))$, where $(q\pm1)_2=2^n$. There is exactly one $G$-conjugacy class of involutions in $P$, represented by $z$.  Moreover, \smallskip $l(B_0(kG))=3$.
\end{enumerate}
In fact, if $P$ is a $2$-subgroup of $G$, but not necessarily a Sylow $2$-subgroup, and $\mathcal{F}_P(G)$ is saturated, then $\mathcal{F}_P(G)$ is isomorphic to one of the fusion systems in (1), (2), and (3).
We refer the reader to  \cite[\S 7.7]{Gor68}, \cite[\S VII]{Bra66} and \cite[Theorem 5.3]{CG12} for these results.\\

\begin{Lemma}\label{PGL}
Let $G:={\mathrm{PGL}}_2(q)$ for a prime power $q$ such that the  Sylow $2$-subgroups of~$G$ are dihedral of order at least $8$, and let $H\leq G$ be a subgroup isomorphic to $\PSL_2(q)$.
Moreover, let $Q\in \Syl_2(H)$ and $P\in\Syl_2(G)$ such that $P\cap H=Q$. Without loss of generality we may set $P:=\langle s,t \mid  s^{2^{n-1}}=t^2=1, tst=s^{-1} \rangle$ and $Q:=\langle s^2,t \rangle$. Then the following holds:
\begin{enumerate}
\item[\rm(a)] $st$ is an involution in $P\setminus Q$, and moreover, any two involutions in $P\setminus Q$ are \smallskip $P$-conjugate.
\item[\rm(b)] Set $z:=s^{2^{n-2}}$. Then centralisers of involutions in $P$ and $G$ are given as \smallskip  follows:
\begin{itemize}
\item[\rm(i)] $C_P(z)=P$, $C_P(t)=\left<t,z \right>\cong C_2\times C_2$ and \smallskip $C_P(st)=\left<st, z \right>\cong C_2\times C_2$.
\item[\rm(ii)] If $q\equiv 1\pmod{4}$, then $C_G(t)\cong D_{2(q-1)}$ and \smallskip  $C_G(st)\cong D_{2(q+1)}$.
\item[\rm(iii)] If $q\equiv -1\pmod{4}$,  then $C_G(t)\cong D_{2(q+1)}$ and \smallskip  $C_G(st)\cong D_{2(q-1)}$.
\end{itemize}
In particular $C_P(z)\in\Syl_2(C_G(z))$ and $C_P(st)\in\Syl_2(C_G(st))$.
\end{enumerate}
\end{Lemma}

\begin{proof}
By assumption $|G:H|=2$, $P\cong D_{2^n}$ with $n\geq 3$ and $Q\cong D_{2^{n-1}}$. The three $P$-conjugacy classes of involutions in $P$ are 
$$\{s^{2^{n-2}}\}, \{s^{2j}t\mid 0\leq j\leq 2^{n-2}\} \text{ and }\{s^{2j+1}t\mid 0\leq j\leq 2^{n-2}\}\,,$$
where $\{s^{2^{n-2}}\},\{s^{2j}t\mid 0\leq j\leq 2^{n-2}\} \subset Q$ and $\{s^{2j+1}t\mid 0\leq j\leq 2^{n-2}\}\subset P\setminus Q$.
Part~(a) follows. \par
For part (b), (i) is obvious and for (ii) and (iii) we refer to \cite[\S4(B)]{GW62} for the description of centralisers of involutions in $\PGL_2(q)$. 
Now it is clear that $C_P(z)\in\Syl_2(C_G(z))$. If $q\equiv 1\pmod{4}$, then $2\mid\mid q+1$, so that $|C_G(st)|_2=|D_{2(q+1)}|_2=4=|C_P(st)|$, whereas if 
$q\equiv -1\pmod{4}$, then $2\mid\mid q-1$, so that $|C_G(st)|_2=|D_{2(q-1)}|_2=4=|C_P(st)|$. Hence $C_P(st)\in\Syl_2(C_G(st))$.
\end{proof}


\section{Properties of Scott modules}\label{sec:Scott}

\begin{Lemma}\label{p-nilpotent}
Let $G$ and $G'$ be finite $p$-nilpotent groups with a common
Sylow $p$-subgroup~$P$. 
Then 
${\mathrm{Sc}}(G\times G', \, \Delta P)$ induces
a Morita equivalence between $B_0(kG)$ and $B_0(kG')$.
\end{Lemma}

\begin{proof}
Set $M:={\mathrm{Sc}}(G\times G',\,\Delta P)$, 
$B:=B_0(kG)$ and $B':=B_0(kG')$. 
By definition, we have
$$M {\,\Big|\,} {1_B}\cdot kG\otimes_{kP}kG'\cdot {1_{B'}}\,.$$
Since $G$ and $G'$ are $p$-nilpotent $1_B\cdot kG\otimes_{kP}kG'\cdot 1_{B'}= kP\otimes_{kP}kP\cong kP$
as  $(kG,kG')$-bimodules. Now as  $kP$ is indecomposable as a $(kP,kP)$-bimodule, it is also indecomposable as a $(kG,kG')$-bimodule.
Therefore $M\cong kP$ as $(B,B')$-bimodule. 
Therefore $M$ induces a Morita equivalence 
between $B$ and $B'$. 
\end{proof}

\begin{Lemma}\label{Rickard96}\label{ScottBrauer}
Assume that $p=2$.
Let $G$ and $G'$ be two finite groups with a common Sylow $2$-subgroup $P$, and 
assume that  $\mathcal F_P(G)=\mathcal F_P(G')$. Let \,$z$ be an involution in $Z(P)$.
Then
$$ {\mathrm{Sc}}\Big(C_{G}(z)\times C_{G'}(z),\, \Delta P\Big) \,\Big|\,
     \Big( {\mathrm{Sc}}(G\times G', \, \Delta P)\Big)(\Delta \langle z\rangle).
$$ 
\end{Lemma}

\begin{proof}
Since $\mathcal F_P(G)=\mathcal F_P(G')$, we have $\mathcal F_{\Delta P}(G\times G')\cong \mathcal F_P(G)$.
Since $P$ is a Sylow $2$-subgroup of $G$,  $\mathcal F_P(G)$ is saturated, and  therefore
$\mathcal F_{\Delta P}(G\times G')$ is also saturated.
Since $\Delta\langle z\rangle\leq Z(\Delta P)$, $\Delta\langle z\rangle$
is a fully normalised subgroup of $\Delta P$ with respect to 
$\mathcal F_{\Delta P}({G\times G'})$ by definition. 
Thus it follows from \cite[Lemmas 3.1 and 2.2]{IK17} that
$$
{\mathrm{Sc}}(N_{G\times G'}(\Delta \langle z\rangle), \, N_{\Delta P}
(\Delta \langle z\rangle)) \,\Big|\,
{\mathrm{Sc}}(G\times G', \, \Delta P)
{\downarrow}_{N_{G\times G'}(\Delta \langle z\rangle)}.
$$ 
Since $z$ is an involution in $Z(P)$, the above reads
$$
{\mathrm{Sc}}(C_G(z)\times C_{G'}(z), \, \Delta P) \,\Big|\,
{\mathrm{Sc}}(G\times G', \, \Delta P){\downarrow}_{C_G(z)\times C_{G'}(z)} .
$$
Set $M:={\mathrm{Sc}}(G\times G',\, \Delta P)$ and 
$\mathfrak M:={\mathrm{Sc}}(C_G(z)\times C_{G'}(z), \, \Delta P)$.
Since taking the Brauer construction is functorial,
by the above we have  that
$$
\mathfrak M(\Delta\langle z\rangle)\,\Big|\,M(\Delta\langle z\rangle).
$$
Obviously $\Delta\langle z\rangle$ is in the kernel of the 
$(kC_G(z)\times kC_{G'}(z))$-bimodule
$kC_G(z)\otimes_{kP} kC_{G'}(z)$,
and hence it is in the kernel of $\mathfrak M$, so that $\mathfrak M^{\Delta\langle z\rangle}=\mathfrak M$.
Therefore by definition of the Brauer construction, we have
$
\mathfrak M(\Delta\langle z\rangle)=\mathfrak M.
$
Therefore 
$\mathfrak M\mid M(\Delta\langle z\rangle)$.
\end{proof}

The following is well-known, but does not seem to appear in the literature. For completeness we provide a proof, which is due to N. Kunugi.

\begin{Lemma}\label{Kunugi} 
Let $G$ and $G'$ be finite groups with a common Sylow $p$-subgroup $P$ such that $\mathcal F_P(G)=\mathcal F_P(G')$.
Assume further that $M$ is an indecomposable  $\Delta P$-projective  $p$-permutation $k(G\times G')$-module. 
Let $Q$ be a subgroup of $P$. Then the following are equivalent.
\begin{enumerate}
\item[\rm (i)] 
${\mathrm{Sc}}(G',Q)\,|\, (k_G\otimes_{kG}M)$.
\item[\rm (ii)]
${\mathrm{Sc}}(G\times G',\,\Delta Q)\,|\,M$.
\end{enumerate}
\end{Lemma}

\begin{proof} By assumption $M\,|\,k_{\Delta R}{\uparrow}^{G\times G'}$, 
where $\Delta R$ is a vertex of $M$ with $R\leq P$. 
Set $N:=k_{\Delta Q}{\uparrow}^{G\times G'}=kG\otimes_{kQ}kG'$. Thus
$N={\mathrm{Sc}}(G\times G',\,\Delta Q)\oplus N_0$
where $N_0$ is a $k(G\times G')$-module which satisfies 
${\mathrm{Hom}}_{k(G\times G')}(N_0,\, k_{G\times G'})=0$, by definition of a Scott module. 
We  have that
\begin{align*}
k_G\otimes_{kG}N &=k_G\otimes_{kG} (kG\otimes_{kQ}kG')
\\
&=
k_G{\downarrow}_Q{\uparrow}^{G'} = k_Q{\uparrow}^{G'}
\\
&=
{\mathrm{Sc}}(G',Q)\oplus X\,,
\end{align*}
where $X$ is a $kG'$-module such that no Scott module can occur as a direct summand of $X$ thanks to Frobenius Reciprocity. Now, let us consider an arbitrary $k(G\times G')$-module $L$ with $L\,|\,N$.
Then,
\begin{align*}
\dim[{\mathrm{Hom}}_{kG'}(k_G\otimes_{kG}L,\,k_{G'})]
&=
\dim[ {\mathrm{Hom}}_{k(G\times G')}(L,\, k_G\otimes_k k_{G'})]
\\
&= 
\dim[{\mathrm{Hom}}_{k(G\times G')}(L, \, k_{G\times G'})]
\\
&
\leq 
\dim[{\mathrm{Hom}}_{G\times G'}(N, \, k_{G\times G'})] = 1
\end{align*}
by adjointness and Frobenius Reciprocity. Hence we have 

$$
\dim[{\mathrm{Hom}}_{k(G\times G')}(L, \, k_{G\times G'})]=
\begin{cases}
1 & \text{ if } {\mathrm{Sc}}(G\times G',\,\Delta Q)\,|\,L\\
0 & \text{ otherwise.}
\end{cases}
$$
Namely,
\begin{align*}
{\mathrm{Sc}}(G', Q)\,|\, (k_G\otimes_{kG}L)
&\text{ \ \ if and only if \,\,}{\mathrm{Sc}}(G\times G',\, \Delta Q)\,|\,L\,.
\end{align*}
Now, we can decompose $M$ into a direct sum of indecomposable $k(G\times G')$-modules
$$
M= M_1\oplus\cdots\oplus M_s\oplus Y_1\oplus\cdots\oplus Y_t
$$
for integers  $s\geq 1$ and $t\geq 0$ , and where for each $1\leq i\leq s$, $M_i={\mathrm{Sc}}(G\times G',\,\Delta Q_i)$
for some $Q_i\leq P$, and for each $1\leq j\leq t$, 
$Y_j\,{\not\cong}\,{\mathrm{Sc}}(G\times G',\,\Delta S)$ for any $S\leq P$.
Then for each $1\leq i\leq s$, we have
$$
k_G\otimes_{kG}M_i =k_G\otimes_{kG}{\mathrm{Sc}}(G\times G',\,\Delta Q_i)
\,{\Big|}\,k_G\otimes_{kG}(kG\otimes_{kQ_i}kG')
$$
where 
$$k_G\otimes_{kG}(kG\otimes_{kQ_i}kG')= k_G{\downarrow}_{Q_i}{\uparrow}^{G'}
= k_{Q_i}{\uparrow}^{G'} \,=\,{\mathrm{Sc}}(G',Q_i)\oplus W_i$$
for a $kG'$-module $W_i$ such that 
${\mathrm{Hom}}_{kG'}(k_{G'},\, W_i)=0$
by Frobenius Reciprocity. 
Now $\mathcal F_P(G)=\mathcal F_P(G')$ implies that $\mathcal F_P(G)=\mathcal F_P(G')\cong\mathcal F_{\Delta P}(G\times G')$.
This means that for $1\leq j,j'\leq s$, we have that $Q_j$ is $G'$-conjugate to $Q_{j'}$  if and only if 
$\Delta Q_j$  is ${(G\times G')}$-conjugate to $\Delta Q_{j'}$.
Therefore the claim follows directly  from the characterisation of Scott modules in \cite[Corollary 4.8.5]{NT88}. 
\end{proof}

\begin{Lemma}\label{Sc-stabEq}
Let $G$ and $G'$ be finite groups with a common Sylow $p$-subgroup $P$ such that $\mathcal F_P(G)=\mathcal F_P(G')$.
Set $M:={\mathrm{Sc}}(G\times G',\,\Delta P)$,
$B:=B_0(kG)$ and $B':=B_0(kG')$.
If 
$M$ induces a stable equivalence of Morita type between
$B$ and $B'$, then the following holds:
\begin{enumerate}
\item[\rm (a)] 
$k_G\otimes_{B}M = k_{G'}$.
\item[\rm (b)]
If  $U$ is an indecomposable $p$-permutation $kG$-module
in $B$ with vertex $1\,{\not=}\,Q\leq P$, then $U\otimes_{B}M$ has, up to isomorphism, a unique
indecomposable direct summand~$V$ with vertex $Q$ and $V$ is again a $p$-permutation module.
\item[\rm (c)]
For any $Q\leq P$,
${\mathrm{Sc}}(G,\,Q)\otimes_{B}M = {\mathrm{Sc}}(G',\,Q)\oplus\text{\rm{(proj)}}$.
\item[\rm (d)]
For any $Q\leq P$,
$\Omega_Q(k_G)\otimes_{B}M = \Omega_Q(k_{G'})\oplus\text{\rm{(proj)}}$.
\end{enumerate}
\end{Lemma}

\begin{proof}{\ }
(a) Apply Lemma~\ref{Kunugi} to the case that $Q=P$. Then Condition (ii) is trivially satisfied. Thus 
we have $k_{G'}\,|\,(k_G\otimes_B M)$, since ${\mathrm{Sc}}(G',P)=k_{G'}$. Hence
Theorem \ref{thm:LinckMoritaEq}(a) yields \smallskip  $k_{G'} = k_G\otimes_{B}M$.

\smallskip

(b) Let $U$ be an indecomposable $p$-permutation $kG$-module
with vertex $Q$.
Since  $M$ induces a stable equivalence of Morita type between $B$ and $B'$, $U\otimes_{B}M$ has a unique non-projective indecomposable  direct
summand, say $V$. Then
\begin{align*}
V\,\Big|\, (U\otimes_{kG}M)&\,{\Big|}\,[U\otimes_{kG}(kG\otimes_{kP}kG')]\\
&=\,
U{\downarrow}_P{\uparrow}^{G'}\,\Big|
\,{k_Q}{\uparrow}^G{\downarrow}_P{\uparrow}^{G'}=\bigoplus_{g\in [Q\backslash G/P]}\ k_{Q^g\cap P}{\uparrow}^{G'}
\end{align*}
by the Mackey decomposition.  Hence $V$ is a $p$-permutation $k{G'}$-module which is 
${(Q^g\cap P)}$-projective for an element $g\in G$.
Thus there is a vertex $R$ of $V$ with $R\leq Q^g\cap P$. 
This means that $gRg^{-1}\leq Q\cap gPg^{-1}\leq Q\leq P$.
Then, since $\mathcal F_P(G)=\mathcal F_P(G')$, there is an element
$g'\in G'$ such that $grg^{-1}=g'r{g'}^{-1}$ for every $r\in R$.
Hence ${g'}R{g'}^{-1}$ is also a vertex of $V$ and 
$g'R{g'}^{-1}=gRg^{-1}\leq Q$, and hence $R \leq_{G'} Q$.

Similarly, we obtain that $Q\leq_G R$ since $M^{\ast}$ induces a stable equivalence of
Morita type between $B'$ and $B$. This implies that $R=_{G'}Q$ and \smallskip $R=_G Q$.

\smallskip

(c) Set $F:=-\otimes_{B}M$, the functor inducing the stable equivalence of 
Morita type between $B$ and $B'$.
Fix $Q$ with $1\,{\not=}\,Q\leq P$, 
and set $U_Q:={\mathrm{Sc}}(G,Q)$. Then, we have
\begin{align*}
1 &=\,\dim [\, {\mathrm{Hom}}_{B}(U_Q,\,k_G)]
\\
&=\, \dim [\, {\underline{\mathrm{Hom}}}_{B}(U_Q,\,k_G)]
\\
&=\, \dim [\, {\underline{\mathrm{Hom}}}_{B'}\Big(F(U_Q),\,F(k_G)\Big)]
\\
&=
\, \dim [\, {\underline{\mathrm{Hom}}}_{B'}\Big(F(U_Q), \, k_{G'}\Big)]
\\
&=
\, \dim [\, {\mathrm{Hom}}_{B'}\Big(F(U_Q), \, k_{G'}\Big)]\,,
\end{align*}
where the second equality holds because $k_G$ is simple, and the last but one equality holds by (a).
Let $V_Q$ be the unique (up to isomorphism) non-projective 
indecomposable direct summand
of $F(U_Q)$. By (b), we know that $Q$ is a vertex of $V_Q$. Moreover, by the above, we have that 
$$\dim[{\mathrm{Hom}}_{B'}(V_Q,\, k_{G'})]\leq 1\,.$$
We claim that in fact equality holds. Indeed, if $\dim[{\mathrm{Hom}}_{B'}(V_Q,\, k_{G'})]=0$, then the above argument implies
that 
$$\dim[{\mathrm{Hom}}_{kG}(F^{-1}(V_Q), \, k_G)]=0\,$$
as well, 
but this is a contradiction
since $U_Q$ is a direct summand of $F^{-1}(V_Q)$, as 
the functor $F$ gives a stable equivalence between $B$ and $B'$.  Hence the dimension is one, and we conclude that $V_Q={\mathrm{Sc}}(G',Q)$.

Next assume  that $Q=1$, so that 
${\mathrm{Sc}}(G,Q)={\mathrm{Sc}}(G,1)=P(k_G)$.
Since
$$P(k_G)\otimes_{kG}M \,\Big|\, P(k_G)\otimes_{kG}kG\otimes_{kP}kG' 
= P(k_G){\downarrow}_P{\uparrow}^{G'}$$
 by definition of a Scott module,
$P(k_G)\otimes_{kG}M$ is a projective $kG'$-module. Moreover, 
it follows from the adjointness and (a) that
$$
\qquad  \dim[{\mathrm{Hom}}_{kG'}(P(k_G)\otimes_{kG}M, \, k_{G'})]
= \dim[{\mathrm{Hom}}_{kG}(P(k_G), \, k_{G'}\otimes_{kG'}M^*)]=1$$
and hence
$P(k_G)\otimes_{kG}M=P(k_{G'})\oplus\mathcal P$
where $\mathcal P$ is a projective $B'$-module 
that does not have $P(k_{G'})$ as a direct summand.  
\smallskip 

(d) Recall that ${\mathrm{Sc}}(G,Q)=P_Q(k_{G})$, ${\mathrm{Sc}}(G',Q)=P_Q(k_{G'})$, and 
$\Omega_Q(k_G)=P_Q(k_{G})/k_G$ and $\Omega_Q(k_{G'})=P_Q(k_{G'})/k_{G'}$. Therefore (d) follows directly from (a), (c)
and the stripping-off method \cite[(A.1) Lemma]{KMN11}.
\end{proof}


\section{Constructing stable equivalences of Morita type}\label{sec:stableeq}


We start with the following gluing result which will allow us to construct stable equivalences of Morita type. It is essentially due to 
Brou{\'e} (\cite[6.3.Theorem]{Bro94}), Rouquier 
(\cite[Theorem 5.6]{Rou01}) and Linckelmann (\cite[Theorem 3.1]{Lin01}). We aim at using equivalence (iii), 
which slightly generalises  the statement of \cite[Theorem 5.6]{Rou01}. 
Since a statement under our hypotheses does not seem to appear in the literature, 
we give a proof for completeness.

\begin{Lemma}
\label{orderP}
Let $G$ and $G'$ be finite groups with a common Sylow $p$-subgroup $P$,
and assume that $\mathcal F_P(G)=\mathcal F_P(G')$.
Set $M:={\mathrm{Sc}}(G\times G',\,\Delta P)$, 
$B:=B_0(kG)$ and $B':=B_0(kG')$.
Further, for each subgroup $Q\leq P$ we set $B_Q:=B_0(kC_G(Q))$
and $B'_Q:=B_0(kC_{G'}(Q))$.
Then, the following three conditions are equivalent.
\begin{enumerate}
\item[\rm(i)] The pair $(M,\,M^*)$ induces a stable equivalence of Morita type
between $B$ and $B'$.
\item[\rm(ii)] For every non-trivial subgroup $Q\leq P$,
the pair $(M(\Delta Q),\, M(\Delta Q)^*)$ 
induces a Morita equivalence between $B_Q$ and $B'_Q$.
\item[\rm(iii)] For every cyclic subgroup $Q\leq P$ of order $p$, 
the pair $(M(\Delta Q),\, M(\Delta Q)^*)$ 
induces a Morita equivalence between $B_Q$ and $B'_Q$.
\end{enumerate}
\end{Lemma}

\begin{proof}
(i)$\Leftrightarrow$(ii) is a special case of \cite[Theorem 3.1]{Lin01}, and (ii)$\Rightarrow$(iii) is trivial.

\smallskip

(iii)$\Rightarrow$(ii):
Take an arbitrary non-trivial subgroup $Q\leq P$. 
Then there is a normal series
\begin{equation*}
C_p\cong Q_1\vartriangleleft Q_2\vartriangleleft \cdots\vartriangleleft Q_m=Q
\end{equation*}
for an integer $m\geq 1$. 
We shall prove (ii) by induction on $m$.
If $m=1$, then $M(\Delta Q)$ induces a Morita equivalence
between $B_Q$ and $B'_Q$ by (iii). 

Next assume that $m\geq 2$ and (ii) holds for $m-1$, and set $R:=Q_{m-1}$. 
Namely, by the inductive hypothesis 
$M(\Delta R)$ realises a Morita equivalence between $B_R$ and $B'_R$.
That is to say,  we have
\begin{equation}\label{LocalMorita}
B_R\cong M(\Delta R)\otimes_{B'_R}\,M(\Delta R)^*.
\end{equation}
Moreover, 
\begin{equation}\label{transitive2}
B_R(\Delta Q)=B_Q
\end{equation}
since $(kC_G(R))(\Delta Q)=kC_G(Q)$
(note that since $R\vartriangleleft Q$, $kC_G(R)$ is a right $k\Delta Q$-module).  
Further,  it follows from \cite[proof of Theorem 4.1]{Ric96} that
\begin{equation}\label{Ric}
\Big( M(\Delta R)\otimes_{B'_R}\,M(\Delta R)^* \Big)(\Delta Q)
\ = \ \left( M(\Delta R)\right)(\Delta Q) \otimes_{B'_Q}\, [\left( M(\Delta R)\right)(\Delta Q)]^*.
\end{equation}
Recall that by \cite[Proposition 1.5]{BP80b} we have
\begin{equation}\label{transitive} 
\left( M(\Delta R)\right)(\Delta Q) = M(\Delta Q).
\end{equation}
Hence 
\begin{align*}  
B_Q &=(B_R)(\Delta Q) 
\quad\text{ by (\ref{transitive2}) }
\\
&= \Big(M(\Delta R)\otimes_{B'_R}\,M(\Delta R)^*\Big)(\Delta Q)
\quad \text{ by (\ref{LocalMorita})}
\\
&=\left( M(\Delta R)\right)(\Delta Q) \otimes_{B'_Q}\, [\left( M(\Delta R)\right)(\Delta Q)]^*                            
\quad\text{ by (\ref{Ric})}
\\
&= M(\Delta Q)\otimes_{B'_Q}\, M(\Delta Q)^*
\quad\text{ by (\ref{transitive})}\,.
\end{align*}
Thus, by making use of \cite[Theorem 2.1]{Ric96}, 
we obtain that  the pair $(M(\Delta Q),\, M(\Delta Q)^*)$ induces a Morita equivalence
between $B_Q$ and $B'_Q$. 
\end{proof}
\bigskip

\noindent \boxed{\text{From now on and until the end of this article we assume that $k$ has 
characteristic $2$.}}\\
\bigskip

The following is an easy application of the
Baer-Suzuki theorem, which is essential to treat dihedral defect groups.

\begin{Lemma}\label{KeyLemma}
Let $G$ be a finite group and let $Q$ be  a normal 
$2$-subgroup of $G$ such that $G/Q\cong \mathfrak S_3$. 
Assume further that there is an involution
$t\in G\setminus Q$. Then $G$ has a subgroup $H$ such that 
$t\in H\cong \mathfrak S_3$.
\end{Lemma}

\begin{proof}
Obviously  $Q=O_2(G)$ since $G/Q\cong\mathfrak S_3$.
Therefore, by the Baer-Suzuki theorem (see \cite[Theorem 3.8.2]{Gor68}),
there exists an element $y\in G$ such that
$y$ is conjugate to $t$ in $G$ and the group $\widetilde{H}:=\langle t,y\rangle$
is not a $2$-group. Therefore $6\mid |\widetilde{H}|$, and since $\widetilde{H}$ 
is generated by two involutions, it is a dihedral group of order $3\cdot 2^a$ for 
some positive integer $a$, 
that is $\widetilde H = C_{3{\cdot}2^{a-1}}\rtimes\langle t\rangle$
(see \cite[Theorem 9.1.1]{Gor68}).  
Seeing  $\widetilde{H}$ as generated by $yt$ and $t$, it 
follows 
immediately that $\widetilde{H}$ has a dihedral subgroup of order $6$, say 
$H$ generated by $t$ and a suitable power of 
$yt$.  The claim follows.
\end{proof}

\begin{Corollary}\label{KeyCor}
Let $G$ be a finite group with a dihedral $2$-subgroup $P$ of order at least $8$, and let $Q\lneq P$ such that 
$Q\cong C_2\times C_2$.
Assume  moreover that $C_G(Q)$ is $2$-nilpotent and that $N_G(Q)/C_G(Q)\cong \mathfrak S_3$.
Then there exists a subgroup $H$ of $N_G(Q)$ such that  $N_P(Q)$ is a Sylow $2$-subgroup of $H$ and $|N_G(Q):H|$ is a power of $2$ (possibly $1$). 
\end{Corollary}

\begin{proof}
Let  $K:=O_{2'}(C_G(Q))$ and let $R\in{\mathrm{Syl}}_2(C_G(Q))$. Then by assumption we  have the following inclusions of subgroups
 $$    \xymatrix@dr@C=1.5pc@R=1.5pc{
 N_G(Q) \ar@{-}[rrrr] \ar@{-}[dd]_{\mathfrak S_3}   & & & &  K\rtimes N_P(Q)=K\rtimes D_8 \ar@{-}[dd]^{C_2}\\
 & & & &\\
 C_G(Q)=K\rtimes R \ar@{-}[rrrr]  & & & &  K\times Q=K\rtimes (C_2\times C_2)
     } $$
where we note that  $|N_P(Q)/Q|=2$ by \cite[Proposition (1B)]{Bra74}, hence $K\rtimes N_P(Q)=K\rtimes D_8$.
Since $[K,\,Q]=1$ by the choice of $K$, 
we have $K\rtimes Q = K\times Q$. Furthermore, as $K$ is a  characteristic subgroup of $C_G(Q)$ and  $C_G(Q)\vartriangleleft N_G(Q)$, we have 
$K\vartriangleleft N_G(Q)$, so that 
$(K\times Q) \vartriangleleft N_G(Q)$.
Clearly $(K\times Q) \, \vartriangleleft \, (K\rtimes N_P(Q))$ 
since the index is two. Also $C_G(Q)\cap (K\rtimes N_P(Q))=K\times Q$.
Therefore we can take quotients by $L:=K\times Q$ of all the groups
in the picture. This yields the following inclusions of subgroups:
 $$    \xymatrix@dr@C=1.5pc@R=1.5pc{
 {\sf G}:=N_G(Q)/L \ar@{-}[rrrr] \ar@{-}[dd]_{\mathfrak S_3}   & & & &  {\sf R}:=(K\rtimes N_P(Q))/L \ar@{-}[dd]^{C_2}\\
 & & & &\\
 {\sf Q}:=C_G(Q)/L \ar@{-}[rrrr]  & & & &  \langle 1\rangle=(K\times Q)/L = {\sf Q}\cap{\sf R}
     } $$
In particular, we have that 
${\sf G}/{\sf Q}\cong \mathfrak S_3$, 
${\sf R}=(K\rtimes N_P(Q))/L\cong N_P(Q)/Q\cong C_2$, 
${\sf Q}=(K\rtimes R)/(K\times Q)\cong R/Q$
(which is a $2$-group), and 
${\sf Q}\cap{\sf R}=\langle 1\rangle$.
Now there must exist an involution ${\sf t}\in{\sf R}$ such that ${\sf t}\not\in{\sf Q}$.
Therefore, by  Lemma \ref{KeyLemma},  there exists
a  subgroup ${\sf H}$ of ${\sf G}$ 
such that ${\sf t}\in{\sf H}\cong  {\mathfrak S}_3$. 
Finally, we set $H$ to be the preimage of 
${\sf H}$ under the canonical homomorphism $N_G(Q)\twoheadrightarrow N_G(Q)/L$ and the 
claim follows.
\end{proof}

\noindent
We can now prove Theorem~~\ref{MainThm} of the Introduction:

\begin{proof}[Proof of Theorem~\ref{MainThm}]
Set $M:={\mathrm{Sc}}(G,P)$. 
Let $Q\leq P$ be an arbitrary fully normalised
subgroup in $\mathcal F_P(G)$. 
We claim that if $Q\neq 1$, then $N_G(Q)$ has a subgroup $H_Q$
which satisfies the conditions (1) and (2) in  Theorem~\ref{thm:IK17Thm1.4}.\par
First suppose that $Q\,{\not\cong}\,C_2\times C_2$. 
Then, ${\mathrm{Aut}}(Q)$ is a 
$2$-group (see 
 \cite[Lemma~5.4.1 (i)-(ii)]{Gor68}),
and hence $N_G(Q)/C_G(Q)$ is also a $2$-group. Thus $N_G(Q)$ is
$2$-nilpotent since $C_G(Q)$ is $2$-nilpotent by the assumption.
Set $N:=N_G(Q)$, and hence we can write $N:=K\rtimes P_N$
where $K:=O_{2'}(N)$ and $P_N$ is a Sylow $2$-subgroup of $N$.
Since $N_P(Q)$ is a $2$-subgroup of $N$, we can assume
$P_N\geq N_P(Q)$. Set $H_Q:=K\rtimes N_P(Q)$. Then, obviously
$N_P(Q)$ is a Sylow $2$-subgroup of $H_Q$ and
$|N:H_Q|$ is a power of $2$ (possibly one) since
$|N:H_Q|=|P_N:N_P(Q)|$. This means that $Q$ satisfies the
conditions (1) and (2) in Theorem~\ref{thm:IK17Thm1.4}.

Next suppose that $Q\cong C_2\times C_2$. Then,
$N_G(Q)/C_G(Q) \hookrightarrow{\mathrm{Aut}}(Q)\cong \mathfrak S_3$.
Clearly $C_P(Q)=Q$ and $N_P(Q)\cong D_8$
(see \cite[Proposition (1B)]{Bra74}). Then, as $N_P(Q)/C_P(Q)\hookrightarrow N_G(Q)/C_G(Q)$, we have that 
$|N_G(Q)/C_G(Q)|\in \{2,6\}$. 
If $|N_G(Q)/C_G(Q)|=2$, then $N_G(Q)$ is $2$-nilpotent
since $C_G(Q)$ is $2$-nilpotent, so that using an argument similar to the one in the previous case there exists a subgroup $H_Q$ of $N_G(Q)$
such that $H_Q$ satisfies the conditions (1) and (2) in Theorem~\ref{thm:IK17Thm1.4}.
Hence we can assume that $|N_G(Q)/C_G(Q)|\,{\not=}\,2$. Now assume that $|N_G(Q)/C_G(Q)|=6$, so that
$N_G(Q)/C_G(Q)\cong \mathfrak S_3$. 
Then it follows from 
Corollary~\ref{KeyCor} that
$N_G(Q)$ has a subgroup $H_Q$ such that
$N_P(Q)$ is a Sylow $2$-subgroup of $H_Q$ and that
$|N_G(Q) : H_Q|$ is a power of $2$, as required. Therefore, by Theorem~\ref{thm:IK17Thm1.4},  $\mathrm{Sc}(N_G(Q),N_P(Q))\!\downarrow^{N_G(Q)}_{QC_G(Q)}$ is indecomposable for each fully normalised
subgroup   $1\neq Q\leq P$.\par
Now if $Q=1$, then $\mathrm{Sc}(N_G(Q),N_P(Q))\!\downarrow^{N_G(Q)}_{QC_G(Q)}=\mathrm{Sc}(G,P)$, which is indecomposable by definition.\par
Therefore  Theorem \ref{thm:IK17Thm1.3} yields that $M$ is Brauer indecomposable. 
\end{proof}

\begin{Corollary}\label{GxG'}
Let  $G$ and $G'$ be finite groups with a common Sylow $2$-subgroup $P$ which is a dihedral group of order
at least $8$ and assume that $\mathcal F_P(G)=\mathcal F_P(G')$.  
Then the Scott module ${\mathrm{Sc}}(G\times G', \Delta P)$ is Brauer indecomposable.
\end{Corollary}

\begin{proof}
Set $M:={\mathrm{Sc}}(G\times G', \Delta P)$. 
Since $P$ is a Sylow $2$-subgroup of $G$, $\mathcal F_P(G)$
is saturated by \cite[Proposition 1.3]{BLO03}. 
Therefore, as $\mathcal F_{\Delta P}(G\times G')\cong \mathcal F_P(G)$ 
by definition, we have that $\mathcal F_{\Delta P}(G\times G')$ is saturated. 

Now let  $\mathcal Q\leq\Delta P$ be any fully normalised subgroup 
in $\mathcal F_{\Delta P}(G\times G')$. Obviously we can write
$\mathcal Q=:\Delta Q$ for a subgroup $Q\leq P$.
We claim that the Brauer construction $M(\Delta Q)$
is indecomposable as a  $k[C_{G\times G'}(\Delta Q)]$-module. Notice that clearly $C_{G\times G'}(\Delta Q)=C_{G}(Q)\times C_{G'}(Q)$.\par

If $Q=1$, then the claim is obvious 
since $M(\Delta\langle 1\rangle)=M$. So, assume that $Q\,{\not=}\,1$. Hence $Q$ contains an involution $t$.
Thus, $C_G(Q) \leq C_{G}(t)$, so that \cite[Lemma (7A)]{Bra66}
implies that $C_{G}(Q)$ is $2$-nilpotent, and similarly for $C_{G'}(Q)$.
Hence $C_{G\times G'}(\Delta Q)$ is $2$-nilpotent.
Therefore it follows from Theorem~\ref{MainThm1} that
$M$ is Brauer indecomposable. 
\end{proof}

\begin{Lemma}\label{BrauerConstruction}
Assume that $G$ and $G'$ are finite groups with a common 
Sylow $2$-subgroup~$P$ which is a dihedral group of order at least $8$. Assume, moreover, that $\mathcal F_P(G)=\mathcal F_P(G')$ and  is such that there are exactly two $G$-conjugacy classes of involutions in $P$.  
Further suppose that $z\in Z(P)$ and $t\in P$ are two involutions in $P$  which  are not $G$-conjugate. 
Set $M:={\mathrm{Sc}}(G\times G',\,\Delta P)$. 
Then
$$
M(\Delta\langle t\rangle)
\cong {\mathrm{Sc}}(C_G(t)\times C_{G'}(t),\,\Delta C_P(t))$$
 as $k[C_G(t)\times C_{G'}(t)]$-modules. 
\end{Lemma}

\begin{proof}
Set $\mathcal F:=\mathcal F_{\Delta P}(G\times G')$. 
By the beginning of the proof of Lemma \ref{ScottBrauer},
$\mathcal F$ is saturated. Set $Q:=\langle t\rangle$.

Next we claim that $Q$ is a fully $\mathcal F_P(G)$-normalised subgroup of $P$.
Let $R\leq P$ be $\mathcal F_P(G)$-conjugate to $Q$. 
So that $R$ and $Q$ are $G$-conjugate. 
So, we can write $R:=\langle r\rangle$ for an element $r\in R$ since $|Q|=2$.
Thus, $t$ and $r$ are $G$-conjugate, which implies that $r$ and $z$ are
not $G$-conjutate by the assumption, and hence $r$ and $t$ are $P$-conjugate
again by the assumption. Namely, $t=r^\pi$ for an element $\pi\in P$.
Obviously, $C_P(r)^\pi = C_{P^\pi}(r^\pi) = C_P(r^\pi)=C_P(t)$, so that
$|C_P(r)|=|C_P(t)|$, which yields that $|N_P(R)|=|N_P(Q)|$ since $|R|=|Q|=2$.
Therefore $Q$ is a fully $\mathcal F_P(G)$-noramlised subgroup of $P$.

Thus, $\Delta Q$ is a fully $\mathcal F$-normalised subgroup of $\Delta P$
since $\mathcal F\cong \mathcal F_P(G)=\mathcal F_P(G')$
(see the beginning of the proof of Lemma \ref{ScottBrauer}).
 Moreover, $M$ is Brauer indecomposable by Lemma~\ref{GxG'}.
Hence it follows from \cite[Theorem 3.1]{IK17} that
$$
M(\Delta Q) \cong {\mathrm{Sc}}(N_{G\times G'}(\Delta Q),\,N_{\Delta P}(\Delta Q))$$
as $k[N_{G\times G'}(\Delta Q)]$-modules.  Noting that $|\Delta Q|=2$ and $\Delta Q=\Delta\langle t\rangle$, we have that
$$
N_{G\times G'}(\Delta Q)=C_{G\times G'}(\Delta Q)
=C_G(Q)\times C_{G'}(Q) = C_G(t)\times C_{G'}(t)
$$
and that
$$
N_{\Delta P}(\Delta Q)=C_{\Delta P}(\Delta Q)= \Delta C_P(Q)=\Delta C_P(t).
$$
The assertion follows.
\end{proof}

\begin{Proposition}\label{stableEquiv}
Let  $G$ and $G'$ be finite groups with a common Sylow $2$-subgroup $P$ which is a dihedral group of order
at least $8$. Assume moreover that $O_{2'}(G)=O_{2'}(G')=1$ and  $\mathcal F_P(G)=\mathcal F_P(G')$.   
Then the Scott module ${\mathrm{Sc}}(G\times G', \Delta P)$   induces a stable equivalence of Morita type 
between the principal blocks $B_0(kG)$ and $B_0(kG')$.
\end{Proposition}

\begin{proof}
Fix $P=D_{2^n}$ for an $n\geq 3$, and set $M:={\mathrm{Sc}}(G\times G', \Delta P)$, $B:=B_0(kG)$ and $B':=B_0(kG')$.  Since we assume that $O_{2'}(G)=O_{2'}(G')=1$, $G$ and $G'$ are amongst the groups (D1)-(D3) listed in \S\ref{sec:2.4}. 
First we note that $G=P$ is the unique group in this list with $\cF_P(G)=\cF_P(P)$, therefore we may assume that $G\neq P\neq G'$.
Thus, by~\S\ref{subsec:fusion} and by the assumption that $\cF_P(G)=\cF_P(G')$, we have that $P$ has the same number of $G$-conjugacy and $G'$-conjugacy classes of involutions, namely either one or two.

Let $t\in P$ be an arbitrary involution, and set
$B_t:=B_0(kC_G(t))$ and $B_t':=B_0(kC_{G'}(t))$. We claim that $M(\Delta\langle t\rangle)$ induces a Morita equivalence
between $B_t$ and $B'_t$.

First of all, assume that $t\in Z(P)$. Set $z:=t$.
Thus $C_P(z)=P$ and this is a Sylow $2$-subgroup of $C_G(z)$, 
so that $P$ is a Sylow $2$-subgroup of $C_{G'}(z)$ as well.  Recall that $C_{G}(t)$ and $C_{G'}(t)$ are both $2$-nilpotent
by \cite[Lemma (7A)]{Bra66}.
Set 
$$M_z:={\mathrm{Sc}}(C_G(z)\times C_{G'}(z), \, \Delta P).
$$
By Lemma~\ref{p-nilpotent},
$M_z$
induces a Morita equivalence between
$B_z$ and $B'_z$. Hence we have
$M_z\,|\, M(\Delta\langle z\rangle)$
by Lemma~\ref{ScottBrauer} and  Corollary~\ref{GxG'} yields
\begin{equation*}
M(\Delta\langle z\rangle)=M_z.
\end{equation*}
This means that $M(\Delta\langle z\rangle)$ induces a Morita
equivalence between $B_z$ and \smallskip $B'_z$. \\

Case 1: Assume first that all involutions in $P$ are $G$-conjugate, so that all involutions in $P$ are $G'$-conjugate a well, since $\mathcal F_P(G)=\mathcal F_P(G')$. 
Therefore there exists an element $g\in G$ and an element $g'\in G'$  such that $z^g=t= z^{g'}$.
Thus, by definition of the Brauer construction, we have
\begin{equation*}
M(\Delta\langle t\rangle)=M(\Delta\langle z^g\rangle)
=M((\Delta\langle z\rangle)^{(g,g')})
=M(\Delta\langle z\rangle)^{(g,g')}=(M_z)^{(g,g')}.
\end{equation*}
Moreover, we have 
\begin{align*}
(M_z)^{(g,g')}
&=
\left({\mathrm{Sc}}(C_G(z)\times C_{G'}(z), \, \Delta P)\right)^{(g,g')}
\\
&={\mathrm{Sc}}(C_G(z^g)\times C_{G'}(z^{g'}), \, (\Delta P)^{(g,g')})\\
&= {\mathrm{Sc}}(C_G(t)\times C_{G'}(t), \, \Delta)\,,
\end{align*}
where $\Delta :=\{(\pi^g,\,\pi^{g'})\mid \pi\in P\}\cong P$.
Obviously $P^g$ is a Sylow $2$-subgroup of $C_G(t)$ and $P^{g'}$ is a Sylow
$2$-subgroup of $C_{G'}(t)$.
Further $C_G(t)$ and $C_{G'}(t)$ are $2$-nilpotent. 
Therefore it follows from Lemma~\ref{p-nilpotent} that
$${\mathrm{Sc}}(C_G(t)\times C_{G'}(t),\, \Delta)=M(\Delta\langle t\rangle)$$ 
induces a Morita equivalence between $B_t$ and \smallskip $B'_t$.
\par
Case 2: Assume now  that $P$ has exactly two  $G$-conjugacy classes and hence exactly two  $G'$-conjugacy classes of involutions. 
Then, by \S\ref{subsec:fusion}, $G$ and $G'$ are groups of type (D3)\rm(ii), that is $G\cong\PGL_2(q)\rtimes C_f$ and $G'\cong\PGL_2(q')\rtimes C_{f'}$ for some odd prime powers $q,q'$ and some suitable odd positive integers $f,f'$. 

If $t$ is $G$-conjugate to the central element $z\in Z(P)$, then 
$M(\Delta\langle t\rangle)$ induces a Morita equivalence
between $B_t$ and $B'_t$ by the same argument as in Case~1. Hence we may assume that $t$ is not $G$-conjugate to $z$. We note that by Lemma~\ref{PGL}(a) any two involutions in $P$ which are not $G$-conjugate (resp. $G'$-conjugate) to $z$ are already $P$-conjugate. It follows easily from Lemma~\ref{PGL}(b) that $C_P(t)$ is a Sylow $2$-subgroup of both $C_G(t)$ and $C_G(t')$. Again, because $C_G(t)$ and $C_{G'}(t)$ are both $2$-nilpotent, it follows  from Lemma \ref{p-nilpotent} that
$$M_t:={\mathrm{Sc}}(C_G(t)\times C_{G'}(t),\,\Delta C_P(t))$$
induces a stable equivalence of Morita type between $B$ and $B'$.

On the other hand, Lemma~\ref{BrauerConstruction}  
 implies that 
$$M(\Delta\langle t\rangle)=M_t\,.$$
Therefore 
$M(\Delta\langle t\rangle)$ induces a Morita equivalence between $B_t$ and $B'_t$. Hence the claim holds. \par
Finally Lemma \ref{orderP} yields that $M$ induces a stable equivalence of Morita type
between $B$ and $B'$.
\end{proof}

\begin{Corollary}\label{cor:steqD2n}
Let $G$ and $G'$ be two finite groups with a common Sylow $2$-subgroup $P\cong D_{2^n}$ with $n\geq 3$ and let $M:={\mathrm{Sc}}(G\times G', \Delta P)$.
\begin{enumerate}
\item[\rm(a)] If $G=\PSL_2(q)$ and $G'=\PSL_2(q')$, where $q$ and $q'$ are powers of odd primes  such that $q\equiv q'\pmod{8}$ and $|G|_2=|G'|_2\geq 8$, then $M$ induces a stable equivalence of Morita type between  $B_0(kG)$ and $B_0(kG')$.
\item[\rm(b)] If $n=3$, $G=\fA_7$ and $G'=\PSL_2(q')$, where $q'$ is a power of an odd prime such that $|G'|_2= 8$, then $M$ induces a stable equivalence of Morita type between  $B_0(kG)$ and $B_0(kG')$.
\item[\rm(c)] If $G=\PGL_2(q)$ and $G'=\PGL_2(q')$, where $q$ and $q'$ are powers of odd primes  such that $q\equiv q'\pmod{4}$ and $|G|_2=|G'|_2$, then $M$ induces a stable equivalence of Morita type between  $B_0(kG)$ and $B_0(kG')$.
\end{enumerate}
Furthermore, there exists a stable equivalence of Morita type between $B_0(kG)$ and $B_0(kG')$ if and only if $\cF_P(G)=\mathcal F_P(G')$.
\end{Corollary}

\begin{proof}
Parts (a), (b) and (c)  follow directly from Proposition~\ref{stableEquiv} since in each case 
$\cF_P(G)=\mathcal F_P(G')$. The sufficient condition of the last statement also follows from Proposition~\ref{stableEquiv}
since we have already noticed that inflation induces splendid Morita equivalences (hence stable equivalences of Morita type) between 
$B_0(kG)$ and $B_0(kG/O_{2'}(G))$, resp. between $B_0(kG')$ and $B_0(kG/O_{2'}(G'))$. 
To prove the necessary condition, we recall that the  existence of a stable equivalence of Morita type  between $B_0(kG)$ and 
$B_0(kG')$ implies that 
$$k(B_0(G))-l(B_0(G))=k(B_0(G'))-l(B_0(G'))$$
(see \cite[5.3.Proposition]{Bro94}). 
Using \cite[Theorem (7B)]{Bra66}, we have that $k(B_0(G))=k(B_0(G'))=2^{n-2}+3$, so that we must have 
$l(B_0(G))=l(B_0(G'))$, which in turn forces $\cF_P(G)=\mathcal F_P(G')$; see \S\ref{subsec:fusion}.
\end{proof}



\section{The principal blocks of  $\PSL_2(q)$ and $\PGL_2(q)$}\label{sec:pslpgl}

Throughout this section we assume that $k$ is a field of characteristic $2$. 
We now start to determine when the stable equivalences of Morita type we constructed  in the previous section are actually Morita equivalences, and in consequence splendid Morita equivalences.
We note that these Morita equivalences are known from the work of Erdmann \cite{Erd90} (over $k$) or Plesken \cite[VII]{Ple83} (over complete discrete valuation rings), but the methods used do not prove that they are \emph{splendid} Morita equivalences.\\

We start with the case $\PSL_2(q)$ and we fix the following notation:  we set $\mathbb B(q):=C/Z(\SL_2(q))$, where $C\leq\SL_2(q)$ is the subgroup of upper-triangular matrices. We have $|\PSL_2(q)|=\frac{1}{2}q(q-1)(q+1)$ and $|\mathbb B(q)|=\frac{1}{2}q(q-1)$. Furthermore, the principal $2$-block of $\PSL_2(q)$ contains three simple modules, namely the trivial module and two mutually dual modules of $k$-dimension $(q-1)/2$, which we denote by $S(q)$ and $S(q)^{*}$. (See e.g. \cite[Lemma 4.3 and Corollary 5.2]{Erd77}). Throughout, we heavily rely heavily on Erdmann's computation of the PIMs of $\PSL_2(q)$ in \cite{Erd77}.

\begin{Lemma}\label{ScottPSL}
Let $G:=\PSL_2(q)$, where $q$ is a power of an odd prime, and  let $P$ be a Sylow $2$-subgroup of $G$.  
Then  the Loewy and socle series structure of the Scott module 
with respect to $\mathbb B(q)$ is 
$${\mathrm{Sc}}(G,\,\mathbb B(q))=  
\boxed{
\begin{matrix} k_G\\ S(q) \oplus  S(q)^* \\ k_G
      \end{matrix}    }\,.
$$
\end{Lemma}

\noindent Before proceeding with the proof, we note that in this lemma we allow the Sylow $2$-subgroups to be Klein-four groups as this case will be necessary when dealing with the groups of type $\PGL_2(q)$ and dihedral Sylow $2$-subgroups of order~$8$.

\begin{proof}
Assume first that $q\equiv -1\pmod{4}$. Then $2\nmid | \mathbb B(q)|$, so that 
${\mathrm{Sc}}(G, \mathbb B(q))=P(k_G)$ (see \cite [Corollary 4.8.5]{NT88}). Therefore, by \cite[Theorem~4(a)]{Erd77}, the Loewy and socle structure of the  PIM $P(k_G)={\mathrm{Sc}}(G, \mathbb B(q))$ is as claimed. \par

Assume next that $q\equiv 1\pmod{4}$. 
By \cite[\S3.2.3]{Bon11} the trivial source module 
${k_{\mathbb B(q)}}{\uparrow}^G$ affords the ordinary character
\begin{equation}\label{Borel}
1_{\mathbb B(q)}{\uparrow}^G=1_G+\text{St}_G\,,
\end{equation}
where $\text{St}_G$ denotes the Steinberg character. Therefore ${k_{\mathbb B(q)}}{\uparrow}^G$ is indecomposable and isomorphic to ${\mathrm{Sc}}(G,\,\mathbb B(q))=:X$. 
Then it follows from \cite[Table 9.1]{Bon11} that
$$X=k_G+(k_G+S(q)+S(q)^*)$$ as composition factors.  Since $X$ is an indecomposable
self-dual $2$-permutation $kG$-module, its Loewy and socle structure
is one of:
$$
\boxed{\begin{matrix} k_G\\S(q)\oplus S(q)^*\\k_G\end{matrix}}\,, \qquad
\boxed{\begin{matrix} k_G\\\,S(q)\, \ \\S(q)^*\\k_G\end{matrix}}\,, \qquad
\boxed{\begin{matrix} k_G\\S(q)^*\\ \,S(q)\, \ \\k_G\end{matrix}}\,, \qquad
\boxed{\begin{matrix} k_G\oplus \,S(q)\, \  \\k_G\oplus S(q)^* \end{matrix}}
\text{\qquad or \qquad}
\boxed{\begin{matrix} k_G\oplus S(q)^* \\k_G\oplus \,S(q)\, \  \end{matrix}}\,.
$$
By \cite[Theorem 2(a)]{Erd77}, 
${\mathrm{Ext}}_{kG}^1(S(q), S(q)^*)=0={\mathrm{Ext}}_{kG}^1(S(q)^*,S(q))$,  hence the second and the third cases
cannot occur. \par
Suppose now that the fourth case happens. Then $X$ has a submodule
$Y$ such that $X/Y\cong k_G$. Hence $Y=S(q)+S(q)^*+k_G$ as composition factors.
Then, since \newline ${\mathrm{Ext}}_{kG}^1(S(q),S(q)^*)=0$, $Y$ has the following structure: 
$$Y= \boxed{\begin{matrix} S(q)\\k_G\end{matrix}}\,\oplus S(q)^*$$
Hence $Y$ has a submodule $Z$ with the Loewy and socle structure 
$$Z=\boxed{\begin{matrix} S(q)\\k_G\end{matrix}}\,.$$
Similarly $X$ has a submodule $W$ such that $X/W\cong S(q)$, and
$W=k_G+k_G+S(q)^*$ as composition factors.
By \cite[Theorem 2]{Erd77}, ${\mathrm{Ext}}_{kG}^1(k_G, k_G)=0$.
Therefore $W$ has the following structure:
$$W=\boxed{\begin{matrix} k_G\\ S(q)^*\end{matrix}}\oplus k_G\,,$$
and hence $W$ has a submodule  $U$ with structure
$$U=\boxed{\begin{matrix} k_G\\ S(q)^*\end{matrix}}\,.$$
Since $Z$ and $U$ are submodules of $X$ and $Z\cap U=0$,
we have a direct sum $Z\oplus U$ in $X$. As a consequence $X=Z\oplus U$, which is 
a contradiction since $X$ is indecomposable. Hence the fourth case cannot occur.\par 
Similarly, the fifth case cannot happen. Therefore we must  have that 
$$
{\mathrm{Sc}}(G,\,\mathbb B(q))\ = \ 
\boxed{
\begin{matrix} k_G\\S(q) \oplus  S(q)^* \\ k_G
      \end{matrix}    }
$$
as desired.   
\end{proof}

\begin{Proposition}\label{PSL_2}
Let $G:=\PSL_2(q)$ and $G':=\PSL_2(q')$, where $q$ and $q'$ are powers of odd primes  such that 
$q\equiv q'\pmod{4}$ and $|G|_2=|G'|_2\geq 8$. Let $P$ be a common 
Sylow $2$-subgroup of $G$ and $G'$. 
Then the Scott module ${\mathrm{Sc}}(G\times G',\,\Delta P)$ induces a splendid Morita equivalence 
between $B_0(kG)$and $B_0(kG')$.
\end{Proposition}

\begin{proof}
Set $M:={\mathrm{Sc}}(G\times G',\Delta P)$ and $B:=B_0(kG)$ and
$B':=B_0(kG')$.
First, by Proposition~\ref{cor:steqD2n}(a), $M$ induces a
stable equivalence of Morita type between $B$ and $B'$. We claim that this  is a Morita equivalence. Using Theorem \ref{thm:LinckMoritaEq}(b), it is enough to check that the simple $B$-modules are mapped to the simple $B'$-modules.\par
To start with, by Proposition~\ref{cor:steqD2n}(a) and Lemma~\ref{Sc-stabEq}(a),  we have
\begin{equation}\label{triMod}
k_G \otimes_B M \ = \ k_{G'}\,.
\end{equation}
Next, because $q\equiv q'\pmod{4}$, the Scott modules  ${\mathrm{Sc}}(G,\,\mathbb B(q))$ 
and ${\mathrm{Sc}}(G',\,\mathbb B(q'))$ have a common  vertex $Q$ (which depend on 
the value of $q$ modulo $4$). Therefore, by Lemma~\ref{Sc-stabEq}(c), we have
$${\mathrm{Sc}}(G,\,\mathbb B(q))\otimes_{B}M 
= {\mathrm{Sc}}(G',\,\mathbb B(q'))\oplus\text{\rm{(proj)}}\,.$$
Moreover, as ${\mathrm{Sc}}(G,\,\mathbb B(q))$ and ${\mathrm{Sc}}(G',\,\mathbb B(q'))$ 
are the relative $Q$-projective covers of $k_G$ and~$k_{G'}$ respectively, it follows  from Lemma \ref{ScottPSL} that the socle series of $\Omega_Q(k_G)$ and $\Omega_Q(k_{G'})$ are given by
$$
      \Omega_Q(k_G) \ = \ 
\boxed{
      \begin{matrix} S(q) \oplus  S(q)^* \\ k_G
      \end{matrix}    } \
       \  \text{ and } \ \ 
      \Omega_Q(k_{G'}) \ = \ \boxed{
      \begin{matrix} S(q') \oplus {S(q')}^* \\ k_{G'}
      \end{matrix}    }\,.
$$
Thus, by Lemma \ref{Sc-stabEq}(d),
$$
  \boxed{
      \begin{matrix} S(q) \oplus  S(q)^* \\ k_G
      \end{matrix}    }  \otimes_BM
      =
  \boxed{
      \begin{matrix} S(q') \oplus  {S(q')}^* \\ k_{G'}
      \end{matrix}    }          
 \oplus \text{ (proj) }.
 $$     
Then it follows from (\ref{triMod}) 
and \cite[Lemma A.1]{KMN11} (stripping-off method) that
$$
(S(q)\otimes_B M)\oplus (S(q)^{*}\otimes_B M) =(S(q)\oplus S(q)^*)\otimes_B M =   S(q')\oplus  {S(q')}^* \oplus (\text{proj}).
$$
Thus Theorem \ref{thm:LinckMoritaEq}(a) implies that 
both  (non-projective) simple 
$B$-modules $S(q)$ and  $S(q)^*$ are mapped to a simple $B'$-module. \par
In conclusion, Theorem \ref{thm:LinckMoritaEq}(b) yields that $M$ induces a Morita equivalence. 
As $M$ is a $2$-permutation $k(G\times G')$-module, 
the Morita equivalence induced by $M$ is actually a splendid Morita equivalence, see \S\ref{subsec:PuigEq}. 
\end{proof}

Next we consider the case $\PGL_2(q)$.  We fix a subgroup $H(q)<\PGL_2(q)$ such that $H(q)\cong\PSL_2(q)$ and keep the notation $\mathbb B(q)<H(q)$ as above. 
Furthermore, the principal $2$-block of $\PGL_2(q)$ contains two simple modules, namely the trivial module and a self-dual module of dimension $q-1$, which we denote by $T(q)$.

\begin{Lemma}\label{ScottPGL}
Let $G:=\PGL_2(q)$ where $q$ is a power of an odd prime.  Then the Loewy and socle sturucture of the Scott module 
${\mathrm{Sc}}(G,\,\mathbb B(q))$ with respect to $\mathbb B(q)$ is as follows:
$$
{\mathrm{Sc}}(G,\,\mathbb B(q))\ = \ 
\boxed{
\begin{matrix} 
 k_G  
\\
\boxed  {\begin{matrix} k_G \oplus T(q) \\
                                 T(q) \oplus k_G \end{matrix} 
            }                                   
\\
k_G 
\end{matrix} 
}
$$
\end{Lemma}

\begin{proof}
Write $H:=H(q)$,  $\mathbb B:=\mathbb B(q)$, $Y:={\mathrm{Sc}}(H,\,\mathbb B)$ and  $X:=Y{\uparrow}^{G}$.
Since $|G/H|=2$, Green's indecomposability theorem and Frobenius Reciprocity
imply that $X={\mathrm{Sc}}(G,\,\mathbb B)$.

Let $\chi$ be the ordinary character of $G$ afforded by the $2$-permutation
$kG$-module $X$. It follows from equation~(\ref{Borel}), \cite[Table 9.1]{Bon11} and Clifford theory, that
\begin{equation}\label{character}
\chi \ = \ 1_{\mathbb B}{\uparrow}^H{\uparrow}^G\  = \ 1_G+1'+ \chi_{St_1} + \chi_{St_2}
\end{equation}
where $1_G$ is the trivial character, $1'$ is the non-trivial linear character of $G$ 
and $\chi_{St_i}$ for $i=1,2$ are the two distinct irreducible constituents of $\chi_{St}\!\uparrow^{G}$ of degree $q$. 
Using \cite[Table 9.1]{Bon11} we have that the $2$-modular reduction of $X$ is 
$$X=k_G+k_G+(k_G+T(q))+(k_G+T(q))$$
as composition factors. 
Moroeover, by Proposition~\ref{ScottPSL}, we have that the Loewy and socle series of $Y$ is 
$$Y=\boxed{\begin{matrix} k_H\\S(q)\oplus S(q)^*\\k_H\end{matrix}}\,.$$
\par
Assume first that $q\equiv -1\pmod{4}$. As in the proof of Lemma~\ref{ScottPSL}, 
as $2\nmid\mathbb B|$, we have that
$$
P(k_G)={\mathrm{Sc}}(G, \mathbb B)\,.
$$
Moreover, by Webb's theorem \cite[Theorem E]{Web82}, the heart $\mathcal{H}(P(k_G))$ of $P(k_G)$ is decomposable with precisely two indecomposable summands and by \cite[Lemma 5.4 and Theorem 5.5]{AC86}, these two summands are dual to each other and endo-trivial modules. It follows that $P(k_G)$ must have the following Loewy and socle structure:
$$
{\mathrm{Sc}}(G, \, \mathbb B)
=
P(k_G)= 
\boxed{
\begin{matrix} 
 k_G  
\\
\boxed{\begin{matrix} k_G \\ T(q) \end{matrix}} 
\oplus
\boxed{\begin{matrix} T(q) \\ k_G \end{matrix}} 
\\
k_G
	\end{matrix}}\,.
$$
\par
Next assume that $q\equiv 1\pmod{4}$. 
Since $X=Y{\uparrow}^G$, $T(q)=S(q){\uparrow}^G=S(q)^*{\uparrow}^G$, Frobenius Reciprocity implies that 
 $X/(X\,{\mathrm{rad}}(kG))\cong {\mathrm{soc}}(X)\cong k_G$,
and that $X$ has a filtration by submodules such that
$X\gneq X_1\gneq X_2$ such that 
$$X/X_1\cong\boxed{\begin{matrix} k_G\\k_G\end{matrix}}\,,\quad X_1/X_2\cong T(q)\oplus T(q)\quad\text{ and }\quad X_2\cong\boxed{\begin{matrix} k_G\\k_G\end{matrix}}$$ 
by Proposition~\ref{ScottPSL}.
Since $X$ is a $2$-permutation $kG$-module,
we know by (\ref{character}) and Scott's theorem on the lifting of homomorphisms 
(see \cite[Theorem II.12.4(iii)]{Lan83}) that
$$
\dim[{\mathrm{Hom}}_{kG}(X, \, \boxed{\begin{matrix}k_G\\k_G\end{matrix}}\,)]
\ = 2.
$$
This implies that $X$ has both a factor module and a submodule which have Loewy structure:
$$\boxed{\begin{matrix}k_G\\k_G\end{matrix}}\,.$$
Now we note that $k_G$ occurs exactly once in the second Loewy layer of $X$
as $|G/O^2(G)|=2$, and we have that the Loewy structure of $X$ is of the form
$$
X\ = \
\boxed{
\begin{matrix} 
 k_G  
\\
k_G \ \cdots
\\
\vdots
\\
k_G
	\end{matrix}}$$
so that only $k_G+(2\times T(q))$ are left to determine.  Further we know from Shapiro's Lemma and \cite[Theorem 2]{Erd77}
that ${\mathrm{Ext}}_{kG}^1(T(q),T(q))=0$
and that 
$$\dim[{\mathrm{Ext}}_{kG}^1(k_G,T(q))]=\dim[{\mathrm{Ext}}_{kG}^1(T(q),k_G)]=1\,.$$
Hence, using the above filtration of $X$, we obtain that
the Loewy structure of $X/X_2$ is
\begin{equation}\label{2}
      X/X_2 \ = \ 
\boxed{\begin{matrix}  k_G \\ k_G \oplus T(q) \\ T(q) \end{matrix}}\,.
\end{equation}
Thus $X$ has  Loewy structure 
\begin{equation}\label{X}
\boxed{\begin{matrix} k_G\\k_G \oplus T(q) \\ T(q)\oplus k_G \\ k_G
\end{matrix}}
\quad \text{ or } \quad
\boxed{\begin{matrix} k_G\\k_G \oplus T(q) \\ T(q)\\k_G \\ k_G
\end{matrix}}\,.
\end{equation}
It follows from (\ref{2}) and the self-dualities of $k_G$, $T(q)$ and $X$  that
$X$ has a submodule $X_3$ such that the socle series has the form
\begin{equation}\label{X3}
X_3 \ = \ 
\boxed{\begin{matrix}  T(q) \\ k_G \oplus T(q) \\ k_G \end{matrix}}
\text{\quad (socle series)}.
\end{equation}

First, assume that the second case in (\ref{X}) holds.
Then, again by the self-dualities, the socle series of $X$ has the form
\begin{equation}\label{Xs}
  X \ = \ \boxed
{\begin{matrix} k_G\\k_G\\ T(q)\\k_G\oplus T(q) \\ k_G
\end{matrix}}
\quad\text{ (socle series)}.
\end{equation}
By making use of (\ref{X}) and (\ref{Xs}), we have
\begin{equation}\label{X_case2}
X \ = \ 
\boxed{
\begin{matrix} 
 k_G  
\\
\boxed{\begin{matrix} k_G \ \\ T(q) \ \\ k_G \end{matrix}} 
\oplus
\boxed{\begin{matrix} T(q) \end{matrix}} 
\\
k_G
	\end{matrix}}
\ = \
\boxed{\begin{matrix} k_G\\k_G \oplus T(q) \\ T(q)\\k_G \\ k_G
\end{matrix}} \text{  (Loewy series)}	
\ = \ \ \boxed
{\begin{matrix} k_G\\k_G\\ T(q)\\k_G\oplus T(q) \\ k_G
\end{matrix}}
\text{ (socle series)}.
\end{equation}
It follows from (\ref{X_case2}) that (up to isomorphism) 
there are exactly four factor modules $U_1, \ldots, U_4$ of $X$ such that
$U_i/(U_i\,{\mathrm{rad}}(kG))\cong{\mathrm{soc}}(U_i)\cong k_G$ for each $i$, and furthermore that
these have structures such that
\begin{equation}\label{Ui}
U_1=X, \ \ U_2\cong k_G, \ \ U_3\cong \boxed{\begin{matrix} k_G \ \\ k_G\end{matrix}}, \ \ 
U_4\cong \boxed{\begin{matrix} k_G \ \\ k_G\ \\ T(q) \ \\ k_G \end{matrix}}.
\end{equation}
Now we know by (\ref{X_case2}) that $X$ has the following socle series
\begin{equation}\label{socleX}
X=
\boxed{\begin{matrix}
k_G \ \\ k_G \ \\ T(q) \ \\ k_G\oplus T(q) \ \\ k_G
\end{matrix}} \ \ \text{(socle series)}.
\end{equation}
If $X$ has a submodule isomorphic to $U_4$, then 
this contradicts (\ref{socleX})
by comparing the third (from the bottom) socle layers of $U_4$ and $X$,
since ${\mathrm{soc}}^3(U_4)\cong k_G$ and 
${\mathrm{soc}}^3(X)\cong T(q)$
(since $U_4$ is a submodule of $X$, ${\mathrm{soc}}^3(U_4) \hookrightarrow{\mathrm{soc}}^3(X)$
by \cite[Chap.I Lemma 8.5(i)]{Lan83}).
This yields that such a $U_4$ does not exist as a submodule of $X$. 
Hence, by (\ref{Ui}), 
$$ \dim_k[{\mathrm{End}}_{kG}(X)]=3.$$
However, by (\ref{character}) and 
Scott's theorem on lifting of endomorphisms of $p$-permutation modules 
\cite[Theorem II 12.4(iii)]{Lan83}, this dimension has to be $4$, so that we 
have a contradiction. As a consequence  the second case in (\ref{X}) 
does not occur. This implies that only the first case in (\ref{X}) can occur. The claim follows.
\end{proof}

\begin{Proposition}\label{PGL_2}
Let $G:=\PGL_2(q)$ and $G':=\PGL_2(q')$, where $q$ and $q'$ are powers of odd 
primes   such that $q\equiv q'\pmod{4}$ and $|G|_2=|G'|_2$. Let $P$ be a common Sylow $2$-subgroup of $G$ and $G'$. 
Then the Scott module ${\mathrm{Sc}}(G\times G',\, \Delta P)$ 
induces a splendid Morita equivalence between $B_0(kG)$ and $B_0(kG')$.
\end{Proposition}

\begin{proof} 
Set $B:=B_0(kG)$, $B':=B_0(kG')$, and 
$M:={\mathrm{Sc}}(G\times G',\,\Delta P)$.
By Proposition~\ref{cor:steqD2n}(c), $M$ induces a stable equivalence of Morita type
between $B$ and $B'$. Again we claim that this stable equivalence is a Morita equivalence.  
Let $Q$ be a Sylow $2$-subgroup of $\mathbb B(q)$.
Then it follows from Lemma \ref{ScottPGL} that
the Loewy structures of $\mathcal H:=\mathcal H(P_Q(k_G)):=\Omega_Q(k_G)/k_G$ and $\mathcal H':=\mathcal H(P_Q(k_{G'})):=\Omega_Q(k_{G'})/k_{G'}$  are given by 
$$
    \mathcal H= 
    \ \boxed{\begin{matrix} k_G\oplus T(q) \\
                                      T(q)\oplus k_G
                 \end{matrix}}
    \quad\text{ and }\quad
\mathcal H'= 
    \ \boxed{\begin{matrix} k_{G'}\oplus T(q') \\
                                      T(q')\oplus k_{G'}
                 \end{matrix}}  \,.
$$
Then it follows from Lemma \ref{Sc-stabEq}(d) that
$$
  \boxed{\begin{matrix} k_G\oplus T(q) \\
                                      T(q)\oplus k_G
                 \end{matrix}}    
 \otimes_B M \ = \ 
\boxed{\begin{matrix} k_{G'}\oplus T(q') \\
                                      T(q')\oplus k_{G'}
                 \end{matrix}} 
 \oplus \text{(proj)}.
 $$
 Thus by the stripping-off method \cite[Lemma A.1]{KMN11} 
and Lemma~\ref{Sc-stabEq}(a)
 we obtain that
 $$(T(q)\oplus T(q))\otimes_B M = T(q')\oplus T(q')\oplus \text{(proj)}.$$
Since $T(q)$ is non-projective, Theorem \ref{thm:LinckMoritaEq}(a) implies that
$$(T(q)\oplus T(q))\otimes_B M = T(q')\oplus T(q').$$
Hence  $T(q)\otimes_B M = T(q')$.
In addition, $k_G\otimes_B M=k_{G'}$ by Lemma \ref{Sc-stabEq}(a).
Finally, since ${\mathrm{IBr}}(B)=\{ k_G, T(q)\}$ and ${\mathrm{IBr}}(B')=\{ k_{G'}, T(q')\}$, it follows from Theorem \ref{thm:LinckMoritaEq}(b)  that
$M$ induces a Morita equivalence between $B$ and~$B'$.
\end{proof}

\section{The principal blocks of $\PSL_2(q)\rtimes C_f$ and $\PGL_2(q)\rtimes C_f$.}\label{sec:autgrps}

Let $q:=r^m$, where $r$ is a fixed odd prime number and $m$ is a positive integer. 
We now let $H$ be one of the groups ${\mathrm{PSL}}_2(q)$ or ${\mathrm{PGL}}_2(q)$, and  we assume, moreover, that a Sylow $2$-subgroup $P$ of $H$ is dihedral of order at least 8.
We let  $G:=H\rtimes C_f$, where $C_f\leq\Gal(\mathbb F_{q}/\mathbb F_{r})$ is as described in cases (D3)(i)-(ii) of Section~\ref{sec:2.4}.  
By the Frattini argument, we have $G=N_G(P)H$, therefore $G/H=N_G(P)H/H$ and we may assume that we have chosen notation such that the cyclic subgroup $C_f$ normalises $P$.

\begin{Lemma}\label{lem:tildeGCGP}
With the notation above, we have $G=H\,C_G(P)$. 
\end{Lemma}

\begin{proof}
The normaliser of $P$ in $G$ has the form
$$N_G(P)=P\,C_G(P)= P\times O_{2'}(C_G(P))$$
because $\Aut(P)$ is a $2$-group and $P$ a Sylow $2$-subgroup of $G$ (see e.g. \cite[Lemma~5.4.1 (i)-(ii)]{Gor68}). 
Therefore, by the above, the subgroup $C_f\leq G$ centralises $P$ 
so that we must have $$G=HC_f \leq H\,C_G(P)$$
and hence equality holds.
\end{proof}

We can now apply the result of Alperin and Dade (Theorem~\ref{thm:AlperinDade}) in order to obtain splendid Morita equivalences.

\begin{Corollary}\label{cor:semidirect}
The principal $2$-blocks $B_0(kG)$ and $B_0(kH)$ are splendidly Morita equivalent.    
\end{Corollary}

\begin{proof}
As $G=H\,C_G(P)$ by Lemma \ref{lem:tildeGCGP} 
the claim follows directly from Theorem~\ref{thm:AlperinDade}.
\end{proof}

\section{Proof of Theorem~\ref{MainTheorem}}\label{sec:proofs}

\begin{proof}[Proof of Theorem~\ref{MainTheorem}]
First because we consider principal blocks only, we may assume that $O_{2'}(G)=1$. 
Therefore, we may assume that $G$ is one of the groups listed in (D1)-(D3) in 
\S\ref{sec:O2'trivial}. 
Now it is known by the work of Erdmann \cite{Erd90} that the principal blocks in 
(1)-(6) fall into distinct Morita equivalence classes.
Therefore the claim follows directly from Corollary~\ref{cor:semidirect}, and 
Propositions~\ref{PSL_2} and \ref{PGL_2}.
\end{proof}

\bigskip
\section{Generalised $2$-decomposition numbers}\label{sec:genDecNum}

Brauer, in \cite[\S VII]{Bra66}, computes character values at $2$-elements for principal blocks with dihedral defect groups up to signs $\delta_1,\delta_2,\delta_3$, thus providing us with the generalised decomposition matrices of such  blocks  up to the signs $\delta_1,\delta_2,\delta_3$.  As a corollary to Theorem~\ref{MainTheorem}, we can now specify these signs.
See also \cite[\S 6]{Mur09} for partial results in this direction.\\

Throughout this section, we assume that $G$ is a finite group with a dihedral Sylow $2$-subgroup 
$P:= D_{2^n}$ of order $2^{n}\geq 8$, for which we use the presentation
$$P=\langle s,t \mid  s^{2^{n-1}}=t^2=1, tst=s^{-1} \rangle\,.$$
Furthermore, we let $\zeta$ denote a primitive $2^{n-1}$-th root of unity in $\mathbb{C}$,
and we let $z:=s^{2^{n-2}}$ (see \S 2.5).

For a $2$-block $B$ of $G$, we let $\mathcal{D}_{\text{gen}}(B)\in \text{Mat}_{k(B)\times k(B)}$ denote its generalised $2$-decom\-position matrix.
In other words: let $S_2(G)$ denote a set of representatives of the $G$-conjugacy classes of the $2$-elements in a fixed defect group  of $B$. Let $u\in S_2(G)$, $H:=C_G(u)$, and  consider $\chi\in\Irr(G)$. Then the \emph{generalised $2$-decomposition 
numbers} are defined to be the uniquely determined algebraic integers  $d^{u}_{\chi\varphi}$ such that
$$\chi(uv)=\sum_{\varphi\in \IBr(H)}d^{u}_{\chi\varphi}\varphi(v)\quad \text{ for }v\in H_{2'}\,,$$
and we set $\mathcal{D}^{u}:=(d^{u}_{\chi\varphi})_{\substack{\chi\in\Irr(B)\\ \varphi\in\IBr(b_u)}}$ where $b_u$ is a $2$-block of $H$ such that $b_u^G=B$,  so that
$$\mathcal{D}_{\text{gen}}(B)=\left(d^{u}_{\chi\varphi}\mid \chi\in\Irr(B), u\in S_2(G), \varphi\in\IBr(b_u) \right)$$
is the \emph{generalised $2$-decomposition matrix} of $B$.  We recall that $\mathcal{D}^{1}$ is simply the $2$-decomposition matrix of $B$. By convention, we see $\mathcal{D}_{\text{gen}}(B)$ as a matrix in $\text{Mat}_{k(B)\times k(B)}$ via $\mathcal{D}_{\text{gen}}(B)=(\mathcal{D}^{u}|u\in S_2(G))$ (one-row block matrix).
Brauer  \cite[Theorem (7B)]{Bra66} proved that $B_0(kG)$ satisfies
$$|\Irr(B_0(kG))|=2^{n-2}+3$$
and possesses exactly $4$ characters of height zero ($\chi_0:=1_G$, $\chi_1$, $\chi_2$, $\chi_3$ in Brauer's notation \cite{Bra66}). In the sequel we always label the 4 first rows of $\mathcal{D}_{\text{gen}}(B_0)$ with these. The remaining characters are of height $1$ and all have the same degree:  unless otherwise specified, we denote them by  $\chi^{(j)}$ with $1\leq i\leq 2^{n-2}-1$, possibly indexed by their degrees.
The first column of $\mathcal{D}_{\text{gen}}(B_0(kG))$ is always labelled with the trivial Brauer character.\\

\begin{Corollary}\label{cor:GenDecNum}
The principal $2$-block $B_0:=B_0(kG)$ of a finite group $G$ with a dihedral Sylow $2$-subgroup 
$P:= D_{2^n}$ of order $2^{n}\geq 8$ affords one of the following generalised decomposition matrices.
\begin{enumerate}
\item[\rm(a)] If $B_0\sim_{SM} B_0(kD_{2^n})$, then
$$
\mathcal{D}_{\text{gen}}(B_0)\,\,=\,\,
\begin{array}{l|c:c:c:r:r}
         &\varphi_{1_1}& z=s^{2^{n-2}}   & s^{a}  &  t & st \\
\hline
1_G& 1& 1& 1& 1&  1\\
\chi_1&1 & 1 & 1 & -1&-1  \\
\chi_2& 1& 1& (-1)^{a}& 1 & -1 \\
\chi_{3}&1 &1 & (-1)^{a} &-1 & 1\\
\chi^{(j)}&2 & 2(-1)^{j} & \zeta^{ja}+\zeta^{-ja} & 0 & 0 \\
\end{array}
$$
where $1\leq j,a\leq 2^{n-2}-1$  \smallskip (up to relabelling the $\chi_i$'s).
\item[\rm(b)] If $n=3$ and $B_0\sim_{SM} B_0(k\mathfrak A_7)$, then 
$$
\mathcal{D}_{\text{gen}}(B_0)\,\,=\,\,
\begin{array}{l|ccc:c:r}
         &\varphi_{1_1}& \varphi_{14_1}& \varphi_{20_1} & z=s^{2} & s\\
\hline
1_G& 1& 0& 0& \ \, \, 1\ &  1\\
\chi_7&1 &1 &0 &-1\  &-1  \\
\chi_8& 1& 0& 1&\ \, \, 1\ &-1  \\
\chi_{9}&1 &1 &1 & -1\ & 1\\
\chi_{5}&0 &1 &0 & \ \, \, 2\  &0  \\
\end{array}
$$
where the irreducible characters and  Brauer characters  are labelled according to the \textsl{ATLAS} \cite{Atlas} and the Modular Atlas \cite{ModAtlas}, \smallskip respectively.
\item[\rm(c)] If $B_0\sim_{SM} B_0(k[{\mathrm{PSL}}_2(q)])$, where  $(q-1)_2=2^n$, then 
$$
\mathcal{D}_{\text{gen}}(B_0)\,\,=\,\,
\begin{array}{l|ccc:c:c}
         &\varphi_{1}& \varphi_{2}& \varphi_{3} & z=s^{2^{n-2}} & s^{a}\\
\hline
1_G& 1& 0& 0& 1&  1\\
\chi_1=\chi_{(q+1)/2}^{(1)}& 1&1 &0 &1 & (-1)^a \\
\chi_2=\chi_{(q+1)/2}^{(2)}&1 &0 & 1& 1 & (-1)^a \\
\chi_3=\chi_{St}& 1&1 &1 &1 &1 \\
\chi_{q+1}^{(j)}&2 & 1& 1& 2(-1)^j  &\zeta^{ja}+\zeta^{-ja}  \\
\end{array}
$$
where   $1\leq j,a\leq 2^{n-2}-1$, $\chi_1$ and $\chi_2$ are labelled by their degrees, and $\chi_{St}$ is the Steinberg \smallskip character.
\item[\rm(d)] If $B_0\sim_{SM} B_0(k[{\mathrm{PSL}}_2(q)])$ with  $(q+1)_2=2^n$, then
$$
\mathcal{D}_{\text{gen}}(B_0)\,\,=\,\,
\begin{array}{l|ccc:c:c}
         &\varphi_{1}& \varphi_{2}& \varphi_{3} & z=s^{2^{n-2}} & s^{a}\\
\hline
1_G& 1& 0& 0& \ \, 1&  1\\
\chi_{(q-1)/2}^{(1)}& 0&1 &0 &-1 & (-1)^{a+1} \\
\chi_{(q-1)/2}^{(2)}&0 &0 & 1&  -1& (-1)^{a+1} \\
\chi_{St}& 1&1 &1 &-1 &-1\ \ \, \\
\chi_{q-1}^{(j)}&0 & 1& 1& 2(-1)^{j+1} & (-1)(\zeta^{ja}+\zeta^{-ja}) \\
\end{array}
$$
where  $1\leq j,a\leq 2^{n-2}-1$, $\chi_1$ and $\chi_2$ are labelled by their degrees, and $\chi_{St}$ is the Steinberg \smallskip character.
\item[\rm(e)] If $B_0\sim_{SM} B_0(k[{\mathrm{PGL}}_2(q)])$  with $2(q-1)_2=2^n$, then
$$
\mathcal{D}_{\text{gen}}(B_0)\,\,=\,\,
\begin{array}{l|cc:r:c:c}
         &\varphi_{1}& \varphi_{2}& t & z=s^{2^{n-2}} & s^{a}\\
\hline
1_G& 1& 0& 1& 1&  1\\
\chi_1=\chi_{q}^{(1)}&1 & 1&-1 & 1& 1 \\
\chi_2=\chi_{q}^{(2)}& 1& 1& 1& 1  & (-1)^{a} \\
\chi_{3}& 1&0 &-1 &1 &(-1)^{a} \\
\chi_{q+1}^{(j)}&2 & 1&0 &  2(-1)^j  &\zeta^{ja}+\zeta^{-ja}   \\
\end{array} 
$$
where  $1\leq j,a\leq 2^{n-2}-1$, $\chi_1$ and $\chi_2$ are labelled by their degrees, and $\chi_3$ is a linear \smallskip character.
\item[\rm(f)] If $B_0\sim_{SM} B_0(k[{\mathrm{PGL}}_2(q)])$ with $2(q+1)_2=2^n$, then 
$$
\mathcal{D}_{\text{gen}}(B_0)\,\,=\,\,
\begin{array}{l|cc:r:c  :c}
         &\varphi_{1}& \varphi_{2}& t & z=s^{2^{n-2}} & s^{a}\\
\hline
1_G& 1& 0& 1&\ \,\, 1&  1\\
\chi_1=\chi_{q}^{(1)}& 1&1 &1 & -1& -1 \ \, \, \\
\chi_2=\chi_{q}^{(2)}&1 &1 & -1& -1  & \ \, \, (-1)^{a+1}  \\
\chi_{3}&1 & 0& -1& \ \,\,1& (-1)^{a} \\
\chi_{q-1}^{(j)}& 0&1 & 0& (-2)(-1)^{j}& (-1)(\zeta^{ja}+\zeta^{-ja}) \\
\end{array} 
$$
where $1\leq j,a\leq 2^{n-2}-1$, $\chi_1$ and $\chi_2$ are labelled by their degrees, and $\chi_3$ is a linear character.
\end{enumerate}
\end{Corollary}

\begin{proof}
Generalised decomposition numbers are determined by a source algebra of the block (see e.g. \cite[(43.10) Proposition]{The95}), hence they are preserved under splendid Morita equivalences. Thus, by Theorem~\ref{MainTheorem}, for a fixed defect group $P\cong D_{2^n}$ ($n\geq 3$), if $n=3$ there are exactly six, respectively,
five if $n\geq 4$, generalised $2$-decomposition matrices corresponding to cases $(a)$ to $(f)$ in Theorem~\ref{MainTheorem}.\par
Let $u\in S_2(G)$ be a $2$-element. First if $u=1_G$, then  by definition $\mathcal{D}^{u}=\mathcal{D}^{1}$ is the $2$-decomposition matrix of $B_0$. Therefore, in all cases, the necessary information about  $\mathcal{D}^{1}$ is given either by Erdmann's work, see \cite[TABLES]{Erd77}, or  the Modular Atlas \cite{ModAtlas}, or \cite[Table 9.1]{Bon11}. It remains to determine the matrices  $\mathcal{D}^{u}$ for $u\neq 1$. As we consider principal blocks only, for each $2$-element $u\in P$, the principal block $b$ of $C_G(u)$ is the unique block of $C_G(u)$ with $b^G=B_0$ 
by Brauer's 3rd Main Theorem \cite[Theorem~5.6.1]{NT88}. 
Moreover, when $P=D_{2^n}$, then centralisers  of non-trivial $2$-elements always possess a normal $2$-complement, so that their principal block is a nilpotent block  \cite[Corollary~3]{Bra64}.
It follows that  $d^{u}_{\chi\,1_H}=\chi(u)$.\par
These character values  are given up to  signs $\delta_1,\delta_2,\delta_3$ by  \cite[Theorem (7B), Theorem (7C), Theorem (7I)]{Bra66}. Thus we can use the character tables of the groups $D_{2^n}$, $\mathfrak A_7$, ${\mathrm{PSL}}_2(q)$ and ${\mathrm{PGL}}_2(q)$ ($q$ odd), respectively, to determine the \smallskip  signs $\delta_1,\delta_2,\delta_3$.
\begin{enumerate}
\item[\rm(a)]    We may assume $G=D_{2^{n}}$. Since $G$ is a $2$-group, the generalised $2$-decomposition matrix  $\mathcal{D}_{\text{gen}}(B_0)$ is the character table of $G$ in this case. The claim \smallskip follows. 

\item[\rm(b)]   We may assume $G=\mathfrak A_7$.  In this case $D=D_8$. Using the ATLAS \cite{Atlas}, we have that the character table of $B_0$ at $2$-elements is 
$$ 
\begin{array}{l|crr}
         &1a &2a &  4a\\
\hline
1_G&  1& 1& 1 \\
\chi_{7}&15  & -1 &-1  \\
\chi_{8}& 21 & 1 & -1 \\
\chi_{9}& 35&-1 &1 \\
\chi_{5}&  14& 2 & 0 \\
\end{array}\,\,.
$$
Hence $\mathcal{D}_{\text{gen}}(B_0)$ follows,  using e.g. the Modular Atlas \cite{ModAtlas} to identify the rows of the matrix.
\item[\rm(c)]  
 We may assume $G={\mathrm{PSL}}_2(q)$ with $(q-1)_2=2^n$.   The height $0$ characters in $B_0$ are $1_G$, the Steinberg character $\chi_{St}$ (of degree $q$), and the two characters of $G$ of degree $(q+1)/2$. 
Because the Steinberg character takes constant value~$1$ on $s^r$ ($1\leq a\leq 2^{n-2}$), we have $\delta_1=1$ in \cite[Theorem (7B)]{Bra66}, so that $\chi^{(j)}(s^a)=\zeta^{ja}+\zeta^{-ja}$ for every  $1\leq a\leq 2^{n-2}$.
And by  \cite[Theorem (7C)]{Bra66}, $\delta_2=\delta_3=-1$, from which 
it 
follows that $\chi_{(q+1)/2}^{(1)}(s^a)=\chi_{(q+1)/2}^{(1)}(s^a)=(-1)^{a}$ for each $1\leq r\leq 2^{n-2}$.
The claim \smallskip follows.
\item[\rm(d)]  
We may assume $G={\mathrm{PSL}}_2(q)$ with $(q+1)_2=2^n$. 
The height $0$ characters in $B_0$ are $1_G$, the Steinberg character $\chi_{St}$ (of degree $q$), and the two characters of $G$ of degree $(q-1)/2$. 
Because the Steinberg character takes constant value $-1$ on $s^a$ ($1\leq a\leq 2^{d-2}$), we have $\delta_1=-1$ in \cite[Theorem (7B)]{Bra66}, so that $\chi^{(j)}(s^a)=(-1)(\zeta^{ja}+\zeta^{-ja})$ for each $1\leq a\leq 2^{d-2}$.
Then by  \cite[Theorem (7C)]{Bra66}, $\delta_2=\delta_3=1$, from which 
it
follows that  $\chi_{(q-1)/2}^{(1)}(s^a)=\chi_{(q-1)/2}^{(1)}(s^a)=(-1)^{a+1}$ for every  $1\leq a\leq 2^{d-2}$.
The claim  \smallskip follows.
\item[\rm(e)]   
 We may assume $G={\mathrm{PGL}}_2(q)$ with $2(q-1)_2=2^n$.  
The height one irreducible characters in $B_0$ have degree $q+1$. The four height zero irreducible characters in $B_0$ 
are the two linear characters $1_G$ and $\chi_{3}$ (in Brauer's notation \cite[Theorem (7I)]{Bra66}) and the two characters of degree $q$, say $\chi_q^{(1)}$ and $\chi_{q}^{(2)}$.\\
Apart from $1_G$, $\chi_q^{(1)}$ is the unique of these taking constant value $1$ on $s^a$ ($1\leq a\leq 2^{d-2}$). Therefore using   \cite[Theorem (7B), Theorem (7C), Theorem (7I)]{Bra66}, we obtain $\delta_1=1$, $\delta_2=-1=\delta_3$. The claim  \smallskip follows.
\item[\rm(f)] 
 We may assume $G={\mathrm{PGL}}_2(q)$ with  $2(q+1)_2=2^n$.  
The height one irreducible characters in $B_0$ have degree $q-1$. The four height zero irreducible characters in $B_0$ are: the two linear characters $1_G$ and $\chi_{3}$ (in Brauer's notation \cite[Theorem (7I)]{Bra66}) and the two characters of degree $q$, say $\chi_q^{(1)}$ and $\chi_{q}^{(2)}$.
Apart from $1_G$, $\chi_q^{(1)}$ is the unique of these taking constant value $1$ on $s^a$ ($1\leq a\leq 2^{n-2}$). Therefore using   \cite[Theorems (7B), (7C) and (7I)]{Bra66}, we obtain $\delta_1=-1$, $\delta_2=1$ and $\delta_3=-1$. The claim  \smallskip follows.
\end{enumerate}
For the character values of $\PSL_2(q)$, we refer to \cite{Bon11} and for the character values of $\PGL_2(q)$, we refer to \cite{Ste51}.
\end{proof}
\bigskip


\section*{Acknowledgments}
\noindent
{\small
The authors thank Richard Lyons
and Ronald Solomon for the proof of Lemma~\ref{KeyLemma}.
The authors are also grateful to Naoko Kunugi
and Tetsuro Okuyama for their useful pieces of advice, and to Gunter Malle for his careful reading of a preliminary version of this work. 
Part of this work was carried out while the first author was visiting 
the TU Kaiserslautern in May and August 2017, who thanks
the Department of Mathematics of  the TU Kaiserslautern for the hospitality.
The second author gratefully acknowledges financial support by the funding body TU Nachwuchsring of the TU Kaiserslautern for the year 2016 when this work started.
Part of this work was carried out during the workshop
"New Perspective in Representation Theory of Finite Groups'' in October 2017 at the Banff International Research Station (BIRS). The authors thank the organisers of the workshop.
Finally, the authors would like to thank the referees for their careful reading and useful
comments.
\bigskip



\end{document}